\documentclass{amsart}
\usepackage{amssymb, latexsym, euscript}
\usepackage{color}

      \def\dC{{\mathbb C}}

   \def\dN{{\mathbb N}}   
      \def\dR{{\mathbb R}}

      \def\cF{{\mathcal F}}

\def\wh#1{{{\widehat #1} }}

\def\bm\chi{\mbox{\boldmath$\chi$}}

\def\RE{{\rm Re\,}}
\def\IM{{\rm Im\,}}

\def\sgn{{\rm sgn\,}}
\def\supp{{\rm supp\,}}

\def\cmr{{\dC \setminus \dR}}

\newcommand{\al}{\alpha}
\newcommand{\be}{\beta}

\newcommand{\ep}{\varepsilon}

\newcommand{\sig}{\sigma}

\renewcommand{\Re}{\hbox{\rm Re }}
\renewcommand{\Im}{\hbox{\rm Im }}

 \DeclareMathOperator{\ii}{i}

\newtheorem{theorem}{Theorem}[section]
\newtheorem{proposition}[theorem]{Proposition}
\newtheorem{corollary}[theorem]{Corollary}
\newtheorem{lemma}[theorem]{Lemma}

\theoremstyle{definition}
\newtheorem{example}[theorem]{Example}
\newtheorem{remark}[theorem]{Remark}

\numberwithin{equation}{section}

\begin{document}
\ifx\JPicScale\undefined\def\JPicScale{1}\fi \unitlength\JPicScale mm

\title[Generalized Nevanlinna functions with one negative square]
{Zeros of nonpositive type of  generalized Nevanlinna functions with one negative square}

\author[H.S.V.~de~Snoo]{Henk de Snoo}
\address{Johann Bernoulli Institute for  Mathematics and Computer Science\\
University of Groningen \\
P.O. Box 407, 9700 AK Groningen \\
Nederland} \email{desnoo@math.rug.nl}

\author[H.~Winkler]{Henrik Winkler}
\address{
Institut f\"{u}r Mathematik\\
Technische Universit\"at Ilmenau \\
Curiebau, Weimarer Str.~25, 98693 Ilmenau \\
Germany} \email{henrik.winkler@tu-ilmenau.de}

\author[M.~Wojtylak]{Micha\l{} Wojtylak}
 \address{
Institute of Mathematics\\
Jagiellonian University\\
 \L ojasiewicza 6\\
 30-348 Krak\'ow\\
Poland} \email{michal.wojtylak@gmail.com}



\thanks{
The first author thanks the Deutsche Forschungsgemeinschaft (DFG) for the Mercator visiting professorship at the Technische Universit\"at Berlin. The third author expresses his gratitude to
the Rijksuniversiteit Groningen and the VU University Amsterdam, where the
research has been partially carried out. He also gratefully acknowledges the assistance of the Polish Ministry of Science and Higher Education grant NN201 546438.}




\begin{abstract}
A generalized Nevanlinna function $Q(z)$ with one negative square has precisely 
one generalized zero of nonpositive type in the closed extended upper halfplane. 
The fractional linear transformation defined by $Q_\tau(z)=(Q(z)-\tau)/(1+\tau Q(z))$, 
$\tau \in \dR \cup \{\infty\}$, is a generalized Nevanlinna function with one negative square.  
Its generalized zero of nonpositive type $\alpha(\tau)$ as a function of $\tau$ defines 
a path in the closed upper halfplane.  Various properties  of this path are studied in detail.
\end{abstract}

\maketitle

\section{Introduction}

The class $\mathbf{N}_1$ of \textit{generalized Nevanlinna functions with one negative square} 
is the set of all scalar functions $Q$ which are meromorphic on $\cmr$, which satisfy
 $Q(\overline{z})=\overline{Q(z)}$, and for which the kernel
\begin{equation}\label{1.0}
\frac{Q(z)-\overline{Q(w)}}{z-\bar{w}},
\end{equation}
has one negative square, see \cite{KL73, KL77, KL81, L}.  The class $\mathbf{N}_0$ of ordinary Nevanlinna
functions consists of functions holomorphic on $\cmr$, which satisfy $Q(\overline{z})=\overline{Q(z)}$, and for which the above
kernel has no negative squares, i.e., $\IM Q(z) / \IM z \ge 0$, $z \in \cmr$; cf. \cite{donoghue}. Let $Q(z)$
belong to  $\mathbf{N}_1$. A point $z_0\in\dC^+\cup\dR\cup\{\infty\}$ is a \textit{generalized zero of
nonpositive type} (\textit{GZNT}) of $Q(z)$  if  either $z_0\in\dC^+$ and
\begin{equation}
Q(z_0)=0,
\end{equation}
or $z_0\in\dR$ and
\begin{equation}\label{x_0}
\lim_{z \wh\to z_0}\frac{Q(z)}{z-z_0}\in (-\infty,0],
\end{equation}
or $z_0=\infty$ and
\begin{equation}\label{x_00}
\lim_{z \wh\to \infty} zQ(z) \in [0,\infty),
\end{equation}
see \cite[Theorem 3.1, Theorem 3.1']{L}.
Here the symbol $\wh\to$ denotes the non-tangential limit.
If $Q(z)$ is holomorphic in a neighborhood of
$z_0\in\dR$, then \eqref{x_0}  simplifies to $Q(z_0)=0$ and
$Q'(z_0)\leq0$. Any function $Q(z) \in \mathbf{N}_1$ has precisely one GZNT in
$\dC^+\cup\dR\cup\{\infty\}$; cf. \cite{KL81}.

Each generalized Nevanlinna function with one negative square is the Weyl function of a closed symmetric
operator or relation with defect numbers $(1,1)$ in a Pontryagin space with one negative square, see
\cite{BDHS,BHSWW,DHS1}. The selfadjoint extensions of the symmetric operator or  relation are parametrized
over $\dR \cup \{\infty\}$; in fact, each selfadjoint extension corresponding to $\tau \in \dR \cup
\{\infty\}$ has a Weyl function of the form:
\begin{equation}\label{bil}
Q_\tau(z)=\frac{Q(z)-\tau}{1+\tau Q(z)}, \quad \tau \in \dR,\quad \mbox{and} \quad Q_\infty(z)=-\frac1{Q(z)}.
\end{equation}
The transform \eqref{bil} takes the class  $\mathbf{N}_1$ onto itself as follows from calculating the kernel
corresponding to \eqref{1.0}. Hence, for each $\tau \in \dR \cup \{\infty\}$ the function $Q_\tau(z)$ has a
unique GZNT, denoted by $\alpha(\tau)$. The study of the path $\tau \mapsto \alpha(\tau)$, $\tau \in \dR \cup
\{\infty\}$, was initiated in \cite{DHS1}. Some simple examples of functions in $\mathbf{N}_1$  may help to
illustrate the various possibilities; cf. Theorem \ref{factor}.

First consider the $\mathbf{N}_1$--function $Q(z)=-z$.  It has a GZNT at the origin. For $\tau \in \dR$ the equation $Q_\tau(z)=0$ has one solution; hence the
path $\alpha(\tau)$ of the GZNT   is given by
\[
 \alpha(\tau)=-\tau, \quad \tau \in \dR.
\]
Therefore  $\alpha(\tau)$ stays on the (extended) real line; cf. Section \ref{ontherealline}.

The function $Q(z)=z^2$ provides another simple example of a function belonging to $\mathbf{N}_1$. Observe that it
has a GZNT at the origin, hence  $\alpha(0)=0$. For $\tau>0$ the equation  $Q_\tau(z)=0$ has two solutions, namely
$-\sqrt{\tau}$ and $\sqrt{\tau}$. Since $Q_\tau'(z)$ exists and is negative (positive) on the negative
(positive) half-axes, it follows that the path $\alpha(\tau)$ of the GZNT is given by
 \[
 \alpha(\tau)=-\sqrt\tau, \quad \tau>0.
 \]
For $\tau<0$ the equation $Q_\tau(z)=0$ has precisely one solution in $\dC^+$,  so that the path
$\alpha(\tau)$ of the GZNT is given by
 \[
 \alpha(\tau)=\ii\sqrt{-\tau}, \quad \tau<0.
 \]
Hence the path approaches the real line
vertically at the origin  and continues along the negative axis; see Figure 1. Finally, note that  $\alpha(\infty)=\infty$.

The function $Q(z)=z^3$ also belongs to $\mathbf{N}_1$ with a GZNT at the origin. For $\tau \in \dR$
the equation $Q_\tau(z)=0$ has three solutions. An argument similar to the one above shows that  the path of
the GZNT of $\alpha(\tau)$ is given by
\[
\alpha(\tau)=\frac{  -\sgn{\tau} +  \sqrt{3}\ii }2\,
 \sqrt[3]{\tau}, \quad \tau\in\dR.
\]
Therefore the path approaches the real line under an angle $\pi/3$ and leaves the real line under an angle
$2\pi/3$; see Figure 1.

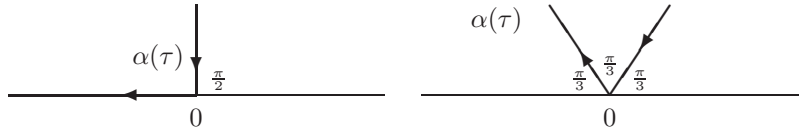
\begin{figure}[htb]
\begin{center}
\begin{picture}(105,20)(0,0)
\linethickness{0.05mm} \put(0,5){\line(1,0){50}}
                       \put(55,5){\line(1,0){50}}
 \thicklines           \put(25,5){\vector(-1,0){10}}
                       \put(15,5){\line(-1,0){15}}
                       \put(25,17){\vector(0,-1){9}}
                       \put(25,14){\line(0,-1){9}}
                       \put(28,7){\makebox(0,0)[cc]{\tiny $\frac{\pi}2$}}
                      \put(20,10){\makebox(0,0)[cc]{$\alpha(\tau)$}}
                        \put(25,02){\makebox(0,0)[cc]{$0$}}
\thicklines         \put(88,17){\vector(-2,-3){4}}
                      \put(84,11){\line(-2,-3){4}}
                      \put(80,5){\vector(-2,3){4}}
                      \put(76,11){\line(-2,3){4}}
                     \put(65,15){\makebox(0,0)[cc]{$\alpha(\tau)$}}
                     \put(80,02){\makebox(0,0)[cc]{$0$}}
                     \put(80,9){\makebox(0,0)[cc]{\tiny $\frac{\pi}3$}}
                    \put(84,7){\makebox(0,0)[cc]{\tiny $\frac{\pi}3$}}
                   \put(76,7){\makebox(0,0)[cc]{\tiny $\frac{\pi}3$}}
\end{picture}
\end{center}
\caption{The path of $\alpha(\tau)$ for $Q(z)=z^2$ and for $Q(z)=z^3$.}
\end{figure}

The previous examples were about functions $Q(z) \in  \mathbf{N}_1$ which have a GZNT on the real axis and
such that $Q(z)$ is holomorphic in a neighborhood of the GZNT. If $Q(z)$ is not holomorphic at its GZNT, the behaviour may be quite different.  For instance,  consider the function $Q(z)=z^{2+\rho}$
where $0<\rho<1$  and the branch is chosen to make $Q(z)$
holomorphic and positive on the positive axis. Then $Q(z)$ belongs to $\mathbf{N}_1$
with a GZNT at the origin.  The path $\alpha(\tau)$ now approaches $z=0$  via the angles
$\pi/(2+\rho)$ and $2\pi/(2+\rho)$, since
\[
 \alpha(\tau)=(-\tau)^{1/(2+\rho)} e^{i \pi /(2+\rho)}, \quad \tau<0, \quad
  \alpha(\tau)=(\tau)^{1/(2+\rho)} e^{i 2\pi /(2+\rho)}, \quad \tau>0,
\]
see Figure 2.

\begin{figure}[hbt]
 \begin{center}
\begin{picture}(100,20)(0,0)
\linethickness{0.05mm} \put(0,5){\vector(1,0){100}}
                       \put(55,4){\line(0,1){2}}
                        \put(55,2){\makebox(0,0)[cc]{$\alpha(0)=0$}}
 \linethickness{0.6mm} \put(0,5){\line(1,0){55}}
                      \put(5,2){\makebox(0,0)[cc]{$\supp\sigma$}}
 \thicklines         \put(63,17){\vector(-2,-3){4}}
                      \put(59,11){\line(-2,-3){4}}
                      \put(55,5){\vector(-2,3){4}}
                       \put(55,5){\line(-2,3){8}}
                     \put(40,15){\makebox(0,0)[cc]{$\alpha(\tau)$}}
                     \put(55,11){\makebox(0,0)[cc]{\tiny $\frac\pi{2+\rho}$}}
                    \put(61,7){\makebox(0,0)[cc]{\tiny $\frac\pi{2+\rho}$}}
\end{picture}
\end{center}
\caption{The path of $\alpha(\tau)$ for $Q(z)=z^{2+\rho}$, $0< \rho <1$. }\label{nons}
\end{figure}
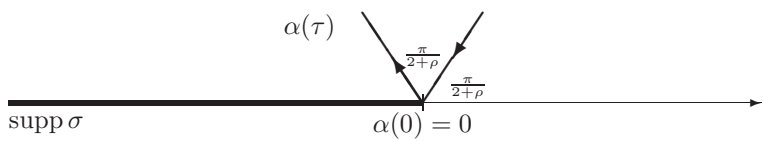

As a final  example,  consider the function
\begin{equation}\label{notonR}
 Q(z)=  \frac{z^2+4}{z^2+1 }\,\,\ii.
\end{equation}
This function belongs to $\mathbf{N}_1$ and it has a GZNT at $z=2\ii$. The equation $Q(z)=\tau$ has a nonreal
solution for each $\tau\in\dR\cup\{\infty\}$. In fact, the path of the GZNT $\alpha(\tau)$ is a simple closed
curve bounded away from the real axis.

In general the path of the GZNT has a complicated behaviour. The path may come from $\dC^+$, be part of the
real line, and then leave again to $\dC^+$, it may approach the real line in different ways, it may stay
completely on the real line, or it may stay away  boundedly   from the real line. In the present paper, some
aspects of the path are treated. Especially, the local behavior of $\alpha(\tau)$ is completely determined in
the domain of holomorphy of the function $Q(z)$. Furthermore under certain holomorphy conditions it is shown
that a small interval of the real line is part of the path of $\alpha(\tau)$. Finally, the case where the
path stays on the extended real line is completely characterized.

The contents of this paper are now described.  Section 2 contains some preliminary observations concerning
generalized Nevanlinna functions with one negative square.  Furthermore a useful version of the inverse
function theorem is recalled.  Some elementary notions concerning the path $\alpha(\tau)$ of the GZNT of
$Q_\tau(z)$ are presented in Section 3. In Section 4 the function $Q(z)$ is assumed to be holomorphic in a
neighborhood of a real GZNT and the path $\alpha(\tau)$ of  the GZNT is studied in such a neighborhood. In
Section 5 necessary and sufficient conditions are given so that a left or right neighborhood of a GZNT of
$Q(z)$ belongs to the path $\alpha(\tau)$. Section \ref{ontherealline} is devoted to the case where the GZNT stays on the
extended real line. A complete characterization is given.

The present paper has points of contact with \cite{JL83,JL95} and \cite{HSSW}.
An example of a path as described in the present paper can be found in
\cite{KL71}.  Furthermore, it should be pointed out that there are strong
connections to the recent perturbation analysis
in \cite{MMRR1,MMRR2,RanWojtylak}.
The authors thank Vladimir Derkach and Seppo Hassi, who have influenced 
this paper in more than one way.

\section{Preliminaries}

\subsection{Nevanlinna functions}

Let $M(z)$ belong to $\mathbf{N}_0$, i.e. $M(z)$ is a  Nevanlinna function (without any negative squares).
Then $M(z)$ has the usual integral representation
\begin{equation}\label{nev'}
M(z)= a+ b z+ \int_\dR \left(\frac{1}{s-z}-\frac{s}{s^2+1}\right) \,d\sigma(s),\quad
z\in\dC\setminus\mathbb{R},
\end{equation}
where $a \in \dR$, $b \ge 0$, and $\sigma$ is a nondecreasing function with
\begin{equation}
\label{int'} \int_\dR \frac{d \sigma(s)}{s^2+1}  < \infty.
\end{equation}
Since the function $\sigma$ can possess jump discontinuities the following normalization is used:
\[
  \sigma(t)=\frac{\sigma(t+0)+\sigma(t-0)}{2}.
\]
In addition, it is assumed that $\sigma(0)=0$. Note that $a$ and $b$ can be recovered from the function
$M(z)$ by
\begin{equation}\label{nev+}
 a=\Re M(i), \quad b = \lim_{ z \wh\to \infty} \frac{M(z)}{z}.
\end{equation}
 Likewise, the function $\sigma$ can be recovered from the
function $M(z)$ by the Stieltjes inversion formula:
\[
  \sigma(t_2)-\sigma(t_1)=\lim_{\varepsilon \downarrow
  0} \frac{1}{\pi} \int_{t_1}^{t_2} \IM
  M(x+i\varepsilon)\,dx, \quad t_1 \leq t_2,
\]
cf. \cite{donoghue}, \cite{KK}.
If $\sigma(s)$ is constant on $(\gamma,\delta)\subseteq\dR$, then $(\gamma,\delta)$  will be called a
\textit{gap} of $\sigma$ or of $M(z)$.  Note that in this case $M(z)$ given by \eqref{nev'} is well--defined
for $z\in(\gamma,\delta)$. By the  Schwarz reflection principle  $M(z)$ is also holomorphic on
$\dC^+\cup\dC^-\cup(\gamma,\delta)$. Conversely, if the function $M(z)$ given by \eqref{nev'} is holomorphic
on some interval $(\gamma,\delta)\subseteq\dR$ then $\sigma(s)$ is constant on that interval. Furthermore,
observe that if $M(z)$ is holomorphic at $z\in\dC$ then
\begin{equation}\label{gapp}
M'(z)=b+\int_{\dR}\frac{d\,\sigma(t)}{(t-z)^2}.
\end{equation}
The symbols $z \downarrow \gamma$ and $z \uparrow \delta$ will stand for  the approximation of $\gamma$ and $\delta$ along $\dR$ from above and below, respectively. So
\begin{equation}\label{rreff}
M(\gamma+)=\lim_{z\downarrow\gamma} M(z) \in [-\infty,
\infty),\quad M(\delta-)=\lim_{z\uparrow\delta}M(z)\in (-\infty, \infty].
\end{equation}
Recall that these limits are equal to the nontangential limits at $\gamma$ and $\delta$, respectively;
see \cite{donoghue}.

The function $\sigma(s)$ introduces in a natural way a measure on $\mathbb{R}$, which is denoted by the same symbol. The formula for the point mass
\begin{equation}\label{nev++}
 \sigma(\{c\})=\sigma(c+0)-\sigma(c-0)= \lim_{z \wh \to c}\, (c-z)M(z),\quad c \in\dR,
\end{equation}
 complements the limit formula in \eqref{nev+}.

The following result is based on a careful analysis of the relationship of the limiting behaviour of the
imaginary part of $M(z)$ and the behaviour of the spectral function $\sigma(s)$ in \eqref{nev'}; see
\cite{donoghue} for details.

\begin{theorem}\label{limits}
Let $M(z)$ be a Nevanlinna function and let $(\gamma,\delta) \subset \dR$ be a finite interval. If
\[
  \lim_{y \downarrow 0} \,\Im M(x+ \ii y)=0,\quad x\in(\gamma,\delta),
\]
then $M(z)$ is holomorphic on $(\gamma,\delta)$.
\end{theorem}

\begin{proof}
Let the function $M(z)$ be of the form \eqref{nev'} and let $x \in (\gamma,\delta)$. It follows from
\eqref{nev'} that
\begin{equation}\label{ilm}
 \IM M(x+iy)=by+\int_\dR \frac{y}{(s-x)^2+y^2}\,d\sigma(s).
\end{equation}
It is known that if the limit of the integral in \eqref{ilm} is $0$ as $y \downarrow 0$, then $\sigma$ is
differentiable at $x$; see \cite[Theorem IV.II]{donoghue}. An application of \cite[Theorem IV.I]{donoghue}
shows that $\sigma'(x)=0$.

Hence, by assumption, it follows that $\sigma$ is differentiable on $(\gamma,\delta)$ and that $\sigma'(x)=0$
for all $x \in (\gamma,\delta)$. Therefore $\sigma$ is constant on $(\gamma,\delta)$ and, hence,  $M(z)$ is
holomorphic on $(\gamma,\delta)$.
\end{proof}

\subsection{Generalized Nevanlinna functions with one negative square}

Assume that $Q(z) \in \mathbf{N}_1$. A point $z_0\in\dC^+\cup\dR\cup\{\infty\}$ is a \textit{generalized pole
of nonpositive type} (GPNT) of $Q(z)$ if $z_0$ is a GZNT for the function $-1/Q(z)$ (which automatically
belongs to $\mathbf{N}_1$). A  function in $\mathbf{N}_1$ has precisely one GPNT in
$\dC^+\cup\dR\cup\{\infty\}$, just as it has precisely one GZNT in $\dC^+\cup\dR\cup\{\infty\}$. For the
following result, see \cite{DHS1,DHS3,DLLSh}.

\begin{theorem}\label{factor}
Any function $Q(z) \in \mathbf{N}_1$ admits the following factorization
\begin{equation}\label{fack}
 Q(z)=R(z) M(z),
\end{equation}
where $M(z) \in \mathbf{N}_0$ and $R(z)$ is a rational function of the form
\begin{equation}\label{einz}
 \frac{(z-\alpha)(z-\overline{\alpha})}{(z-\beta)(z-\overline{\beta })},
 \quad
 \frac{ 1}{(z-\beta)(z-\bar{\beta)}}, \quad
  \mbox{or}
  \quad
 (z-\alpha)(z-\bar{\alpha}).
\end{equation}
Here $\alpha, \beta \in \dC^+ \cup \dR \cup \{\infty\}$ stand for the GZNT and GPNT of $Q(z)$, respectively;
in the first case $\alpha$ and $\beta$ are finite, in the second case $\infty$ is a GZNT and $\beta$ is
finite, and  in the third case $\alpha$ is finite and $\infty$ is a GPNT.
\end{theorem}

For the function  $Q(z) \in \mathbf{N}_1$ the function $M(z) \in \mathbf{N}_0$ and the factors in
\eqref{einz} are uniquely determined. Note that $\al\ne\be$, otherwise $Q(z)$ would not have any negative
squares.

\begin{corollary}\label{factor+}
Let $Q(z) \in \mathbf{N}_1$ and let $z_0 \in \dC^+$. If $Q(z_0)=0$, then $Q'(z_0) \ne 0$.
\end{corollary}

\begin{proof}
Since $Q(z_0)=0$, it follows that $z_0 \in \dC^+$ is a GZNT of $Q(z)$. Therefore, according to Theorem
\ref{factor}, $Q(z)=(z-z_0)(z-\bar{z}_0) H(z)$ where $H(z)$ is holomorphic in a neighborhood of $z_0$ and
$H(z_0) \neq 0$ (since $M(z_0) \neq 0$). Differentiation of the identity  leads to
\[
 Q'(z)=(z-z_0)H(z)+(z-\bar z_0)H(z)+(z-z_0)(z-\bar z_0)H'(z),
\]
which implies that
\[
Q'(z_0)=2 \text{i} \, (\IM z_0)H(z_0).
\]
Since $z_0 \in \dC^+$ and $H(z_0)\neq 0$,  this implies that $Q'(z_0) \ne 0$.
\end{proof}

Let $Q(z) \in \mathbf{N}_1$ and assume that $Q(z)$ is holomorphic in a neighborhood of $z_0 \in \dR$.  The
following is a classification of $Q(z_0)=0$;  see \cite{DHS3}. A proof is included for completeness.

\begin{proposition}\label{zeros>0}
Let $Q(z) \in \mathbf{N}_1$  and assume that $Q(z)$ is holomorphic in a neighborhood of $z_0 \in \dR$.  If
$Q(z_0)=0$, then  precisely one of the following possibilities occurs:
 \begin{itemize}
 \item[(0)]    $Q'(z_0) > 0$;
 \item[(1)]    $Q'(z_0) < 0$;
 \item[(2a)]  $Q'(z_0) = 0$ and $Q''(z_0) > 0$;
 \item[(2b)]  $Q'(z_0) = 0$ and $Q''(z_0) < 0$;
 \item[(3)]    $Q'(z_0) = 0$ and $Q''(z_0) = 0$ (in which case $Q'''(z_0) > 0$).
 \end{itemize}
In the cases {\rm(1)--(3)} the point $z_0 \in \dR$ is necessarily a GZNT of the function $Q(z)$.
\end{proposition}

\begin{proof}
Since $Q(z) \in \mathbf{N}_1$ it is of the form $Q(z)=R(z)M(z)$, where $R(z)$ is of the form  \eqref{einz}
and $M(z)$ is a Nevanlinna function; see Theorem \ref{factor}. Assume that  $Q(z)$ is holomorphic in a
neighborhood of $z_0 \in \dR$ and that $Q(z_0)=0$.

\textit{Case 1}. Consider the case $R(z_0) \neq 0$. Then $M(z)$ must be holomorphic around $z_0 \in \dR$ and
$M(z_0)=0$. Observe that $M'(z_0)>0$, otherwise $M(z)$ would be constant, see relation \eqref{gapp}, and
hence identical zero, which would imply $Q(z)\equiv 0\notin \mathbf{N}_1$.
 It follows from $Q'(z)=R'(z)M(z)+R(z)M'(z)$ that $Q'(z_0)=R(z_0) M'(z_0)$. Observe from
\eqref{einz} that $R(z_0) >0$. Therefore $Q'(z_0)>0$, so that (0) occurs.

\textit{Case 2}. It remains to consider the case $R(z_0) = 0$. This implies that $z_0=\alpha \in \dR$. Hence
\[
 Q(z)=(z-\alpha)^2M(z) \quad \mbox{or} \quad Q(z)=\frac{(z-\alpha)^2}{(z-\beta)(z-\bar{\beta})}\,M(z).
\]
Since $Q(z)$ is holomorphic around $\alpha \in \dR$, it follows that the Nevanlinna function $M(z)$ has the
expansion
\begin{equation}\label{nevv}
M(z)=\frac{m_{-1}}{z-\alpha}+m_0+m_1 (z-\alpha)+ \cdots
\end{equation}
where $m_{-1} \le 0$ and $m_i \in \dR$, $i \in \dN$. Moreover
\[
\begin{split}
 Q(z)&=c_0m_{-1} (z-\alpha)+(c_0m_0+c_1m_{-1})(z-\alpha)^2 \\
 &\hspace{2cm} +(c_0m_1+c_1m_0+c_2m_{-1})(z-\alpha)^3+\cdots,
\end{split}
\]
where $c_0>0$ and $c_i \in \dR$, $i \in \dN$, stand for the coefficients of the power series expansion of the
function $[(z-\beta)(z-\bar{\beta})]^{-1}$ or of the function $1$ if $\beta=\infty$. In particular, the
following identities are clear:
\begin{equation}\label{id1}
 Q'(\alpha)=c_0m_{-1},
\end{equation}
\begin{equation}\label{id2}
2Q''(\alpha)=c_0m_0+c_1m_{-1},
\end{equation}
and
\begin{equation}\label{id3}
 6Q'''(\alpha)=c_0m_1+c_1m_0+c_2m_{-1}.
\end{equation}
It follows from \eqref{id1} that $Q'(\alpha) \le 0$. There is the following subdivision:
\begin{itemize}
\item $Q'(\alpha) < 0$. Then (1) occurs.
\item $Q'(\alpha)=0$ and $Q''(\alpha)>0$ or $Q''(\alpha)<0$. Then (2a) or (2b) occur.
\item $Q'(\alpha)=0$ and $Q''(\alpha)=0$. In this case it follows from \eqref{id1} and \eqref{id2}
that $m_{-1}=0$ and $m_0=0$. Note that \eqref{nevv} with $m_{-1}=m_0=0$ implies that $m_1>0$. According to
\eqref{id3}, it follows that $Q'''(\alpha)=c_0m_1$.
Therefore $Q'''(\alpha) > 0$, so that (3) occurs.
\end{itemize}

To conclude observe that it follows from \eqref{x_0} that $Q(z_0)=0$ and $Q'(z_0) \le 0$  imply that $z_0 \in
\dR$ is a GZNT of $Q(z)$.
\end{proof}

\subsection{The inverse function theorem}

The following consequence of the usual inverse function theorem will be useful. It specifies branches of
solutions of an equation involving holomorphic functions; cf. \cite[Theorem 9.4.3]{H}.

\begin{theorem}\label{inver'}
Let $Q(z)$ be a function which is holomorphic at $z_0$ and assume that  $Q^{(i)}(z_0)=0$, $0 \le i \le n-1$,
and $Q^{(n)}(z_0)\ne 0$ for some $n \ge 1$ so that
\[
 Q(z)=\frac{Q^{(n)}(z_0)}{n!}(z-z_0)^n+\frac{Q^{(n+1)}(z_0)}{(n+1)!}(z-z_0)^{n+1} + \cdots.
\]
Then there is a neighborhood of $z_0$ where the equation
\begin{equation}\label{inv'}
Q(z)=w^n
\end{equation}
has $n$  solutions $z=\phi^{i}(w)$, $1\le i\le n$. The functions $\phi^{i}(w)$ are holomorphic at $0$, of the
form
\[
 \phi^{i}(w)=z_0+\phi^{i}_1 w+ \phi^{i}_2 w^2+\cdots,
\]
and their first order coefficients $\phi^{i}_1$, $1\le i\le n$, are the $n$ distinct roots of the equation
\begin{equation}
\label{crit'} (\phi_1^i)^n= \frac{n!}{Q^{(n)}(z_0)},\quad 1\le i\le n.
\end{equation}
\end{theorem}

\section{Elementary properties of $\alpha(\tau)$ }\label{basicalpha}

Let $Q(z) \in \mathbf{N}_1$. The fractional linear transform $Q_\tau(z)$, $\tau \in \dR \cup \{\infty\}$, is
defined  in \eqref{bil}, so that the derivative  of $Q_\tau$ is given by
\begin{equation}\label{QQ'}
 Q'_\tau(z)=(1+\tau^2)\frac{Q'(z)}{(1+\tau Q(z))^2}, \quad \tau \in \dR, \quad \mbox{and}
 \quad Q'_\infty(z)=\frac{1}{Q(z)^2}.
\end{equation}
Since $Q_\tau(z)$  also belongs to $\mathbf{N}_1$ for $\tau\in\dR\cup\{\infty\}$, Theorem~\ref{factor} may be
applied.

\begin{corollary}
Let $Q(z) \in \mathbf{N}_1$. Then $Q_\tau(z)$, $\tau\in\dR\cup\{\infty\}$, has the factorization
\begin{equation}\label{QQ+}
 Q_\tau(z)=R_{(\tau)}(z) M_{(\tau)}(z),
\end{equation}
where $M_{(\tau)}(z) \in \mathbf{N}_0$ and $R_{(\tau)}(z)$ is a rational function of the form
\begin{equation}\label{einzt}
 \frac{(z-\alpha(\tau))(z-\overline{\alpha(\tau)})}{(z-\beta(\tau))(z-\overline{\beta(\tau}))},
 \quad
 \frac{ 1}{(z-\beta(\tau))(z-\overline{\beta(\tau)})},
 \quad
 \mbox{or}
 \quad
  (z-\alpha(\tau))(z-\overline{\alpha(\tau)}).
\end{equation}
Here $\alpha(\tau), \beta(\tau) \in \dC^+ \cup \dR \cup \{\infty\}$ stand for the GZNT and GPNT of
$Q_\tau(z)$, respectively; in the first case $\alpha(\tau)$ and $\beta(\tau)$ are finite, in the second case
$\infty$ is a GZNT and $\beta(\tau)$ is finite, and in the third case $\alpha(\tau)$ is finite and $\infty$
is a GPNT.
\end{corollary}

Note that $\alpha(0)=\alpha$ and $\beta(0)=\beta$, and that $\alpha(\tau)\ne\be(\tau)$ for all
$\tau\in\dR\cup\{\infty\}$. The paths $\alpha(\tau)$ and $\beta(\tau)$ are related. To see this, observe that
it follows from the form of the fractional linear transform \eqref{bil} that
\[
 Q_{-1/\tau}(z)=\frac{Q(z)+1/\tau}{1-Q(z)/\tau}
 =-Q_\tau(z)^{-1}, \quad \tau \in \dR \cup \{\infty\}.
\]
Therefore, \eqref{QQ+} and \eqref{einzt} lead to
\begin{equation}\label{alphbet}
 \alpha(\tau)=\beta({-1/\tau}), \quad \beta(\tau)=\alpha(-1/\tau),\quad\tau\in\dR\setminus\{0\},
\end{equation}
and, in particular, to $\alpha(\infty)=\beta$ and $\beta(\infty)=\alpha$.

According to the identities in \eqref{alphbet} it suffices to describe the function $\alpha(\tau)$. The
\textit{path} of the GZNT of the function $Q(z) \in \mathbf{N}_1$ is defined by
\[
\mathcal{F}_Q:=\{\,\alpha(\tau):\tau\in\dR\cup\{\infty\}\,\}.
\]
The following result concerning the parametrization of the path may be useful.

\begin{lemma}
Let $Q(z) \in \mathbf{N}_1$ and let $\tau_0 \in \dR \cup \{\infty\}$. Then the path of the GZNT of $Q(z)$
coincides with the path of the GZNT of $Q_{\tau_0}(z)$, i.e.
 \begin{equation}\label{FQt}
\mathcal{F}_Q=\mathcal{F}_{Q_{\tau_0}},\quad \tau_0\in\dR\cup\{\infty\}.
\end{equation}
\end{lemma}

\begin{proof}
First consider the case $\tau \in \dR$. Then a simple calculation shows that
\begin{equation}\label{tr}
 (Q_\tau)_\rho(z)= Q_{\frac{\tau+\rho}{1-\rho\tau}}(z),
 \quad \rho\in\dR,
 \quad (Q_\tau)_\infty(z)=Q_{-1/\tau}(z).
\end{equation}
Hence, if the GZNT of $(Q_\tau)_\rho(z)$ is denoted by $\alpha_\tau(\rho)$,  then it follows from \eqref{tr}
that
\begin{equation}\label{tr1}
\alpha_\tau(\rho)=\alpha\left(\frac{\tau+\rho}{1-\rho\tau}\right), \quad \rho\in\dR,\quad
\alpha_\tau(\infty)=\alpha_{-1/\tau}.
\end{equation}
The identity \eqref{tr1}  shows that  \eqref{FQt} is valid. The case $\tau=\infty$ can be treated similarly.
\end{proof}

The points of $\mathcal{F}_Q$, i.e. the solutions of the equation  $\alpha(\tau_0)=z_0$,   will be
characterized in terms of the function $Q(z)$. But first observe the following. If $Q(z)$ is holomorphic in a
neighborhood of $\alpha(\tau_0)$ and if $Q'(\alpha(\tau_0))\neq0$, then clearly the function $\alpha(\tau)$
is holomorphic in a neighborhood of $\tau_0$; this follows from the usual inverse function theorem (see
Theorem \ref{inver'} with $n=1$).

\begin{theorem}\label{charact}
Let $Q(z) \in \mathbf{N}_1$ and let $\tau_0 \in \dR$.
\begin{enumerate}\def\labelenumi {\rm (\roman{enumi})}
\item Let $z_0 \in \dC^+$. Then $\alpha(\tau_0)=z_0$ if and only if
\begin{equation}\label{un}
    Q(z_0)=\tau_0.
\end{equation}
In this case, the function $\alpha(\tau)$ is holomorphic in a neighborhood of $\tau_0$.

\item  Let $z_0\in\dR$. Then $\alpha(\tau_0)=z_0$ if and only if
\begin{equation}\label{deux}
\lim_{z \wh\to z_0}\frac{Q(z)-\tau_0}{z-z_0} \in (-\infty,0].
\end{equation}

\item  Let $z_0=\infty$. Then $\alpha(\tau_0)=\infty$ if and only if
\begin{equation}\label{deux+}
\lim_{z \wh\to \infty} z (Q(z)-\tau_0) \in [0,\infty).
\end{equation}

\end{enumerate}
\end{theorem}

\begin{proof}
(i)  Let $z_0 \in \dC^+$ and let $\tau_0 \in \dR$.

($\Leftarrow$) If \eqref{un} holds, then $Q_{\tau_0}(z_0)=0$. Hence, $z_0 \in \dC^+$ is a  zero of
$Q_{\tau_0}(z) \in \mathbf{N}_1$, which implies that $z_0$ is a GZNT of $Q_{\tau_0}(z)$. By uniqueness it
follows that $z_0=\alpha(\tau_0)$.

($\Rightarrow$) If $\alpha(\tau_0)=z_0$, then $Q_{\tau_0}(z_0)=Q_{\tau_0}(\alpha(\tau_0))=0$, so that
\eqref{un} holds.

If $z_0:=\alpha(\tau_0)\in\dC^+$, then $Q_{\tau_0}(z_0)=0$ and hence
 $Q'_{\tau_0}(z_0) \ne 0$; cf. Corollary \ref{factor+}.
It follows from \eqref{QQ'} that $Q'(z_0)\ne 0$. Therefore $\alpha(\tau)$ is holomorphic at $\tau_0$.

(ii) Let $z_0 \in \dR$ and let $\tau_0 \in \dR$.
It follows from the fractional linear transform \eqref{bil} that
\begin{equation}\label{deux--}
\frac{Q_{\tau_0}(z)}{z-z_0}=\frac{1}{1+\tau_0 Q(z)} \,\,\frac{Q(z)-\tau_0}{z-z_0}.
\end{equation}

($\Leftarrow$) Assume that \eqref{deux} holds.
Since \eqref{deux} implies $\lim_{z\wh\to z_0} Q_{\tau_0}(z)=0$, it follows from \eqref{deux--} that
\[
 \lim_{z \wh\to z_0}\frac{Q_{\tau_0}(z)}{z-z_0}
 =\frac1{1+\tau_0^2}\,\lim_{z \to z_0} \frac{Q(z)-\tau_0}{z-z_0} \in (-\infty,0].
\]
Hence, by \eqref{x_0} $z_0$ is the GZNT of $Q_{\tau_0}(z)$.

($\Rightarrow$)
Since $Q_{\tau_0}(z)/(z-z_0) \to 0$ as $z \wh\to z_0$, it follows that $Q_{\tau_0}(z) \to 0$ as $z \wh\to z_0$ and, therefore, $Q(z) \wh\to \tau_0$. Hence $\alpha(\tau_0)=z_0$ follows from \eqref{deux--}.

(iii) Let $z_0 = \infty$ and let $\tau_0 \in \dR$.
It follows from the fractional linear transform \eqref{bil} that
\begin{equation}
z Q_{\tau_0}(z) =\frac{1}{1+\tau_0 Q(z)} \,\,z( Q(z)-\tau_0).
\end{equation}
The proof now uses \eqref{x_00} with similar arguments  as in (ii).
\end{proof}

The case $\tau_0=\infty$ is not explicitly mentioned in Theorem \ref{charact}.
Recall that $\alpha(\infty)=\beta$. Hence, the identity $z_0=\alpha(\infty)$ means actually that $z_0$ is a generalized pole of nonpositive type of the function $Q(z)$.

\begin{corollary}\label{quh}
Let $Q(z) \in \mathbf{N}_1$, then
\[
\{\,z\in \dC^+ : \IM Q(z)=0\,\} \subseteq\mathcal{F}_Q.
\]
\end{corollary}

\begin{proof}
Let $z_0 \in \dC^+$ and assume that $\IM Q(z_0)=0$. Then
\[
 Q(z_0)-\tau_0= \RE Q(z_0)-\tau_0,
\]
so that the lefthand side equals zero, if $\tau_0$ is defined as $\RE Q(z_0)$. In this case $Q_{\tau_0}(z)$
has a zero at $z_0$.
\end{proof}

\begin{corollary}\label{alphainjective}
Let $Q(z) \in \mathbf{N}_1$, then the function
\[
\dR\cup\{\infty\}\ni\tau\mapsto \alpha(\tau)\in\dC^+\cup\dR\cup\{\infty\}
\]
is injective.
\end{corollary}

\begin{proof}
Assume that $\alpha(\tau_1)=\alpha(\tau_2)$ with $\tau_1, \tau_2 \in \dR \cup \{\infty\}$.

First consider $\tau_1, \tau_2 \in \dR$ and let $z_0= \alpha(\tau_1)=\alpha(\tau_2)$.
If $z_0$ is in $\dC^+$, then (i) of Theorem \ref{charact}
implies $\tau_1=Q(\alpha(\tau_1)=Q(\alpha(\tau_2)=\tau_2$.
If $z_0$ is in $\dR\cup\{\infty\}$, then
(ii) and (iii) of Theorem \ref{charact} imply
\[
\tau_1=\lim_{z\wh\to \alpha(\tau_1)} Q(z)=\lim_{z\wh\to \alpha(\tau_2)} Q(z) =\tau_2.
\]

Next consider $\tau_1 \in \dR$ and $\tau_2=\infty$ and let $z_0= \alpha(\tau_1)=\alpha(\infty)$.
Then $\alpha(\infty)=z_0$ means that $z_0$ is a GPNT, so that $Q(z) \to \infty$ as $z \wh\to z_0$.  Furthermore, $\alpha(\tau_1)=z_0$ implies, by Theorem \ref{charact}),
that $Q(z) \to \tau_1$ as $z \wh\to z_0$, a contradiction.

Hence, $\alpha(\tau_1)=\alpha(\tau_2)$ with $\tau_1, \tau_2 \in \dR \cup \{\infty\}$, implies that
$\tau_1=\tau_2$.  This completes the proof.
\end{proof}

The following result can be seen as a consequence of Theorem \ref{limits}.

\begin{theorem}\label{tauonreal}
Let $Q(z) \in \mathbf{N}_1$ and let the interval $(\gamma,\delta) \subset \dR$ be contained in
$\mathcal{F}_Q$.   Then $Q(z)$ is holomorphic on $(\gamma,\delta)$ except possibly at the GPNT of $Q(z)$,
which is then a pole of $Q(z)$.
\end{theorem}

\begin{proof}
Since $Q(z) \in \mathbf{N}_1$, it can be written as $Q(z)=R(z)M(z)$ as in \eqref{fack} where $R(z)$ is of the
form as in \eqref{einz} with GZNT $\alpha$ and GPNT $\beta$. By assumption each $z_0 \in (\gamma,\delta)$ is
of the form $z_0=\alpha(\tau_0)$. Hence by Theorem \ref{charact} it follows that
\begin{equation}\label{frid}
 \lim_{z\wh\to z_0} Q(z)=\tau_0.
\end{equation}
Clearly, if $z_0=\alpha(\tau_0)$ with $\tau_0=\infty$ then $z_0=\beta$ so that $z_0$ is a GPNT of $Q(z)$.
There are three cases to consider:

\textit{Case 1:} $\beta \not\in (\gamma,\delta)$ and $\alpha \not\in (\gamma,\delta)$. Then  $\lim_{z\wh\to
z_0} R(z) \in \dR\setminus \{0\}$, so that it follows from \eqref{frid} that
\begin{equation}\label{frid1}
 \lim_{z\wh\to z_0} \IM M(z)=0.
\end{equation}
Hence, by Theorem \ref{limits}, it follows from \eqref{frid1}  that $M(z)$ and therefore $Q(z)$ is
holomorphic on $(\gamma,\delta)$.

\textit{Case 2:}  $\beta \not\in (\gamma,\delta)$ and $\alpha \in (\gamma,\delta)$. Then by Case~1 it follows
that  $Q(z)$ is holomorphic on $(\gamma,\alpha)$ and on $(\alpha, \delta)$. Hence, either $Q(z)$ is
holomorphic on $(\gamma,\delta)$ or $Q(z)$ has an isolated singularity at $\alpha$. However, this last case
cannot occur due to the representation \eqref{einz}.

\textit{Case 3:} $\beta \in (\gamma,\delta)$. Then by Case 1 and Case 2 it follows that $Q(z)$ is holomorphic
on $(\gamma,\beta)$ and $(\beta,\delta)$. This implies that $\beta$ is an isolated singularity of $Q(z)$; in
other words the GPNT $\beta$ is a pole of $Q(z)$.
\end{proof}

\section{Local behavior of $\alpha(\tau)$ in a gap of $Q(z)$.}\label{Gap}

Let the function  $Q(z) \in \mathbf{N}_1$ be holomorphic in a neighborhood of  $z_0 \in \dR$. Assume that
$Q(z_0)=0$ and that $z_0$ is in fact a GZNT of $Q(z)$, so that $Q'(z_0) \le 0$; cf. Proposition
\ref{zeros>0}. The local form of the path $\alpha(\tau)$ in a neighborhood of $\tau=0$ will now be described.
The items in the following theorem correspond to the classification in Proposition~\ref{zeros>0}.

\begin{theorem}\label{mainth}
Let  $Q(z) \in \mathbf{N}_1$ be holomorphic in a neighborhood of $z_0 \in \dR$ and let $z_0$ be a GZNT of
$Q(z)$. Then precisely one of the following possibilities hold.
\begin{itemize}
\item[(1)]  $Q'(z_0)<0$:
There exists $\varepsilon > 0$ such that the function $\alpha(\tau)$ is real-valued and holomorphic with
$\alpha'(\tau)<0$ on $(-\varepsilon, \varepsilon)$.

\begin{figure}[htb]

\begin{center}
\begin{picture}(100,12)(0,0)
\linethickness{0.05mm} \put(0,5){\vector(1,0){100}}
                       \put(55,4){\line(0,1){2}}
                       \put(55,2){\makebox(0,0)[cc]{$\alpha(0)$}}
 \linethickness{0.6mm} \put(0,5){\line(1,0){30}}
                      \put(70,5){\line(1,0){20}}
                      \put(15,2){\makebox(0,0)[cc]{$\supp\sigma$}}
 \thicklines          \put(60,5){\vector(-1,0){15}}
                       \put(45,5){\line(-1,0){5}}
                       \put(45,8){\makebox(0,0)[cc]{$\alpha(\tau)$}}
\end{picture}
\caption{Case (1)}
\end{center}
\end{figure}

\item[(2a)] $Q'(z_0)=0$ and $Q''(z_0) >0$:
There exist $\varepsilon_1 > 0$ and $\varepsilon_2 > 0$ such that the function $\alpha(\tau)$ is continuous
on $(-\varepsilon_1,\varepsilon_2)$, and holomorphic on each of the intervals $(-\varepsilon_1,0)$ and
$(0,\varepsilon_2)$. Moreover $\alpha(\tau) \in \dC^+$ for $\tau \in (-\varepsilon_1,0)$ and $\arg
(\alpha(\tau)-z_0)\to \pi/2$ as $\tau \uparrow 0$ and $\alpha(\tau) \in \dR$ for $\tau \in
(0,\varepsilon_2)$.

\item[(2b)] $Q'(z_0)=0$ and $Q''(z_0) < 0$:
There exist $\varepsilon_1 > 0$ and $\varepsilon_2 > 0$ such that the function $\alpha(\tau)$ is continuous
on $(-\varepsilon_1,\varepsilon_2)$, and holomorphic on each of the intervals $(-\varepsilon_1,0)$ and
$(0,\varepsilon_2)$. Moreover $\alpha(\tau) \in \dR$ for $\tau \in (-\varepsilon_1,0)$ and $\alpha(\tau) \in
\dC^+$ for $\tau \in (0,\varepsilon_2)$ and $\arg (\alpha(\tau)-z_0)\to \pi/2$ as $\tau \downarrow 0$.

\begin{figure}[htb]
\begin{center}
\begin{picture}(105,22)(0,0)
\linethickness{0.05mm} \put(0,5){\line(1,0){50}}
                       \put(55,5){\line(1,0){50}}
                       \linethickness{0.6mm}
                       \put(35,5){\line(1,0){15}}
                      \put(45,2){\makebox(0,0)[cc]{$\supp\sigma$}}
                      \put(90,5){\line(1,0){15}}
                      \put(100,2){\makebox(0,0)[cc]{$\supp\sigma$}}
 \thicklines          \put(25,5){\vector(-1,0){10}}
                       \put(15,5){\line(-1,0){5}}
                     \put(25,11){\vector(0,-1){3}}
                       \put(25,14){\line(0,-1){9}}
                       \put(20,14){\makebox(0,0)[cc]{$\alpha(\tau)$}}
                       \put(25,2){\makebox(0,0)[cc]{$\alpha(0)$}}
                        \put(28,7){\makebox(0,0)[cc]{\tiny $\frac{\pi}2$}}
                        \qbezier(25,14)(25,17)(26,20)

                       \put(80,5){\vector(-1,0){10}}
                       \put(70,5){\line(-1,0){5}}
                       \put(65,9){\vector(0,1){3}}
                       \put(65,14){\line(0,-1){9}}
                       \put(70,14){\makebox(0,0)[cc]{$\alpha(\tau)$}}
                       \put(65,2){\makebox(0,0)[cc]{$\alpha(0)$}}
                        \put(68,7){\makebox(0,0)[cc]{\tiny $\frac{\pi}2$}}
                       \qbezier(65,14)(65,17)(66,20)
\end{picture}
\caption{Cases (2a) and (2b)}
\end{center}
\end{figure}
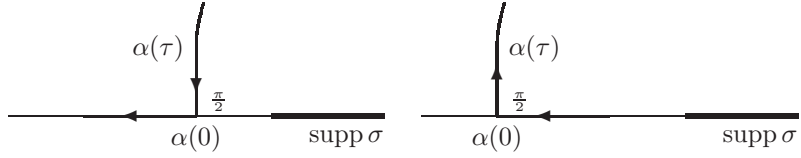

\item[(3)]
$Q'(z_0)=Q''(z_0)=0$,  and $Q'''(z_0)>0$: There exist $\varepsilon_1 > 0$ and $\varepsilon_2 > 0$ such that
the function $\alpha(\tau)$ is continuous on $(-\varepsilon_1,\varepsilon_2)$, and holomorphic on each of the
intervals $(-\varepsilon_1,0)$ and $(0,\varepsilon_2)$. Moreover $\alpha(\tau) \in \dC^+$ for $\tau \in
(-\varepsilon_1,0)$ and $\arg (\alpha(\tau)-z_0)\to \pi/3$ as $\tau \uparrow 0$; and $\alpha(\tau) \in \dC^+$
for $\tau \in (0,\varepsilon_2)$ and $\arg (\alpha(\tau)-z_0)\to 2\pi/3$ as $\tau \downarrow 0$.
\begin{figure}[hbt]
\begin{center}
\begin{picture}(100,20)(0,0)
\linethickness{0.05mm} \put(0,5){\vector(1,0){100}}
                       \put(55,4){\line(0,1){2}}
                        \put(55,2){\makebox(0,0)[cc]{$\alpha(0)$}}
 \linethickness{0.6mm} \put(0,5){\line(1,0){30}}
                      \put(70,5){\line(1,0){20}}
                      \put(15,2){\makebox(0,0)[cc]{$\supp\sigma$}}
 \thicklines         \put(63,17){\vector(-2,-3){4}}
                      \put(59,11){\line(-2,-3){4}}
                      \put(55,5){\vector(-2,3){4}}
                      \qbezier(51,11)(48,17)(42,18)
                     \put(40,15){\makebox(0,0)[cc]{$\alpha(\tau)$}}
                     \put(55,9){\makebox(0,0)[cc]{\tiny $\frac{\pi}3$}}
                    \put(59,7){\makebox(0,0)[cc]{\tiny $\frac{\pi}3$}}
                   \put(51,7){\makebox(0,0)[cc]{\tiny $\frac{\pi}3$}}
\end{picture}
\caption{Case (3) }
\end{center}
\end{figure}

\end{itemize}
\end{theorem}

\begin{proof}
The assumption is that $Q(z) \in \mathbf{N}_1$ is holomorphic in a neighborhood of $z_0 \in \dR$ and that
$Q(z)$, possibly together with some derivatives,  vanishes at $z_0 \in \dR$, as described in Proposition
\ref{zeros>0}. The theorem will be proved via Theorem \ref{inver'} (implicit function theorem).

\textit{Case (1): $Q(z_0)=0$ and $Q'(z_0)<0$.} According to Theorem \ref{inver'} with $n=1$  there is some neighborhood of $w=0$ where the equation
\begin{equation}\label{Qphi}
 Q(\phi(w))=w
\end{equation}
has a unique holomorphic solution $\phi(w)$ which satisfies $\phi(0)=z_0$ and is real-valued for real $w$. It
follows from \eqref{Qphi} that $Q'(\phi(w))\phi'(w)=1$, so that the condition $Q'(z_0)<0$ implies that
$\phi'(0)<0$, and thus $\phi'(w)<0$ on some neighborhood $(-\varepsilon, \varepsilon)$.  It follows from
Theorem \ref{charact} that $\alpha(\tau)=\phi(\tau)$.

\textit{Case (2a): $Q(z_0)=0$, $Q'(z_0)=0$, and $Q''(z_0) >0$.} According to Theorem \ref{inver'} with $n=2$
 there is some neighborhood of $0$ where the equation
\[
Q(\phi^\pm(w))=w^2
\]
has holomorphic solutions $\phi^+(w)$ and $\phi^-(w)$ with $\phi^\pm(0)=z_0$, and which have expansions
\[
\phi^\pm(w)=\phi_1^\pm w+\phi_2^\pm w^2+ \cdots ,
\]
where $\phi_1^\pm=\pm (Q''(z_0)/2)^{-1/2}$. Note that all the coefficients $\phi_i^\pm$ in the above
expansions are real.

Let $\tau>0$. Put $w=\tau^{1/2}$ such that
\begin{equation}\label{pmtau}
Q(\phi^\pm(\tau^{1/2}))=\tau.
\end{equation}
Since $\phi_1^-<0$, there is some $\varepsilon_2>0$ such that $\phi^{-'}(\tau^{1/2})<0$ for
$\tau\in(0,\varepsilon_2)$. Now the relation \eqref{pmtau} implies by taking the derivative with respect to
$\tau$ that $Q'(\phi^-(\tau^{1/2}))<0$ for $\tau\in(0,\varepsilon_2)$. Hence, by relation \eqref{pmtau} and
Theorem \ref{charact} one finds that $\alpha(\tau)=\phi^-(\tau^{1/2})$ for $\tau\in(0,\varepsilon_2)$.

Let $\tau<0$. Put $w=i|\tau|^{1/2}$ such that
\begin{equation}\label{pmtaui}
Q(\phi^\pm(i|\tau|^{1/2}))=\tau.
\end{equation}
Since $\phi_1^+>0$, there is some $\varepsilon_1>0$ such that $(\phi^{+})'(i|\tau|^{1/2})\in\dC^+$ for
$\tau\in(-\varepsilon_1,0)$. Hence, by relation \eqref{pmtaui} and Theorem \ref{charact} one finds that
$\alpha(\tau)=\phi^+(i|\tau|^{1/2})$ for $\tau\in(-\varepsilon_1,0)$.

The expansion of $\phi^+(i|\tau|^{1/2})$ implies that $\arg (\alpha(\tau)-z_0)\to \pi/2$ as $\tau\uparrow 0$.

\textit{Case (2b): $Q(z_0)=0$, $Q'(z_0)=0$, and $Q''(z_0)<0$.} This case can be treated similarly as the case
(2a).

\textit{Case (3): $Q(z_0)=0$, $Q'(z_0)=0$, $Q''(z_0)=0$, and $Q'''(z_0) > 0$.} According to Theorem
\ref{inver'} there is a neighborhood of $0$, where the equation
$$
Q(\phi(w))=w^3
$$
 has three solutions $\phi^{(j)}(w)$, $j=1,2,3$,  determined by
\[
\phi^{(1)}_1=\sqrt[3]{r},\quad
 \phi^{(2)}_1=\left(-\frac12+\frac{\sqrt{3}}2\ii\right)\sqrt[3]{r},\quad
  \phi^{(3)}_1 = \left(-\frac12-\frac{\sqrt{3}}2\ii\right)\sqrt[3]{ r},
\]
with  $r=6/Q'''(z_0)$.

Let $\tau>0$. Then  Theorem \ref{inver'} implies that there is some $\varepsilon_2>0$ such that
$\phi^{(2)}(\tau^{1/3})$ is in $\dC^+$ for $\tau\in( 0,\varepsilon_2)$, and that
$Q(\phi^{(2)}(\tau^{1/3}))=\tau$. Hence, by Theorem \ref{charact}
\[
\alpha(\tau)=\phi^{(2)}(\tau^{1/3}),\qquad \tau\in(0,\varepsilon_2).
\]

Let $\tau<0$. Then the Theorem \ref{inver'} implies that there is some $\varepsilon_1>0$ such that
$\phi^{(3)}(-|\tau|^{1/3})$ is in $\dC^+$ for $\tau\in (-\varepsilon_1,0)$, and that
$Q(\phi^{(3)}(-|\tau|^{1/3}))=-|\tau|=\tau$. Hence, by Theorem \ref{charact}
\[
\alpha(\tau)=\phi^{(3)}(-|\tau|^{1/3})),\qquad \tau\in(-\varepsilon_1,0)<0.
\]
Moreover, using Theorem \ref{inver'} one finds that
\[
\lim_{\tau\uparrow0}\tan(\arg(\alpha(\tau)-z_0))=\lim_{\tau\uparrow0}\frac{\IM\alpha(\tau)}{\RE\alpha(\tau)-z_0}=
 \frac{\IM\phi_1^{(3)}}{\RE\phi_1^{(3)}}
 =\sqrt{3},
\]
and
\[
\lim_{\tau\downarrow0}\tan(\arg(\alpha(\tau)-z_0))=\lim_{\tau\downarrow0}\frac{\IM\alpha(\tau)}{\RE\alpha(\tau)-z_0}=
 \frac{\IM\phi_1^{(2)}}{\RE\phi_1^{(2)}}
 =-\sqrt{3},
\]
which shows that the angles are indeed $\pi/3$ and $2\pi/3$, respectively.
\end{proof}

\begin{remark}\label{R2}
In Case (2) and Case (3) other zeros of $Q_\tau(z)$ will occur, but they need not be GZNT.
Below each case will be considered separately.

For Case (2) it suffices to restrict to Case (2a) as Case (2b) is similar. When $\tau>0$ the
function $\alpha_+(\tau):=\phi^+(\tau^{1/2})$ is also a local solution of $Q(\alpha_+(\tau))=\tau$. However,
$Q'(\alpha_+(\tau))>0$ for $\tau>0$, so that $\alpha_+(\tau)$ is not a GZNT. Moreover, $\alpha_+(\tau)$ is
locally increasing in $\tau$. A different situation happens when $\tau<0$. Then clearly $\phi^-(-\ii
\tau^{1/2})=\phi^+(\ii \tau^{1/2})$, and locally $\bar\alpha(\tau)=\phi^+(-\ii \tau^{1/2})$ is the GZNT of
$Q_\tau(z)$ in the lower half plane conjugate to the GZNT $\alpha(\tau)$, cf \cite{KL71}.

\begin{figure}[htb]
\begin{center}
\begin{picture}(105,25)(0,10)
\linethickness{0.05mm} \put(0,20){\line(1,0){50}}
                       \put(55,20){\line(1,0){50}}
 \thicklines \put(25,20){\vector(-1,0){10}}
                       \put(15,20){\line(-1,0){5}}
                       \put(25,20){\vector(1,0){10}}
                       \put(35,20){\line(1,0){5}}
                       \put(25,26){\vector(0,-1){3}}
                       \put(25,29){\line(0,-1){9}}
                       \put(25,14){\vector(0,1){3}}
                       \put(25,11){\line(0,1){9}}
\put(20,29){\makebox(0,0)[cc]{$\alpha(\tau)$}} \put(13,17){\makebox(0,0)[cc]{$\alpha(\tau)$}}
\put(37,17){\makebox(0,0)[cc]{$\alpha_+(\tau)$}} \put(30,11){\makebox(0,0)[cc]{$\bar\alpha(\tau)$}}
                       \put(65,20){\vector(1,0){10}}
                       \put(75,20){\line(1,0){5}}
                       \put(95,20){\vector(-1,0){10}}
                       \put(85,20){\line(-1,0){5}}
                       \put(80,24){\vector(0,1){3}}
                       \put(80,29){\line(0,-1){9}}
                       \put(80,16){\vector(0,-1){3}}
                       \put(80,11){\line(0,1){9}}
\put(75,29){\makebox(0,0)[cc]{$\alpha(\tau)$}} \put(68,17){\makebox(0,0)[cc]{$\alpha_+(\tau)$}}
\put(92,17){\makebox(0,0)[cc]{$\alpha(\tau)$}} \put(85,11){\makebox(0,0)[cc]{$\bar\alpha(\tau)$}}
\end{picture}
\caption{Cases (2a) and (2b) - positive and negative zeros}
\end{center}
\end{figure}
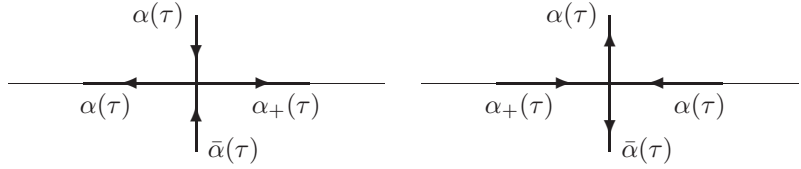

In Case (3) there are locally two other zeros of the function $Q_\tau(z)$ crossing $z_0$ at $\tau=0$. The
first of them is real and given by
\[
  \alpha_+(\tau)=\phi^{(1)}(\sgn\tau |\tau|^{1/3}).
\]
The second one lies in the lower half plane and is conjugate to $\alpha(\tau)$.

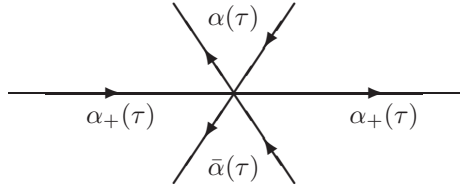
\begin{figure}[htb]
\begin{center}
\begin{picture}(100,27)(0,0)
\linethickness{0.05mm} \put(20,12){\line(1,0){60}}
 \thicklines\put(58,24){\vector(-2,-3){4}}
                      \put(54,18){\line(-2,-3){4}}
                      \put(50,12){\vector(-2,3){4}}
                      \put(46,18){\line(-2,3){4}}
                      \put(50,22){\makebox(0,0)[cc]{$\alpha(\tau)$}}
                       \put(58,0){\vector(-2,3){4}}
                       \put(54,6){\line(-2,3){4}}
                       \put(50,12){\vector(-2,-3){4}}
                       \put(46,6){\line(-2,-3){4}}
                      \put(50,2){\makebox(0,0)[cc]{$\bar\alpha(\tau)$}}
                        \put(25,12){\vector(1,0){10}}
                        \put(35,12){\vector(1,0){35}}
                        \put(70,12){\line(1,0){5}}
                        \put(35,9){\makebox(0,0)[cc]{$\alpha_+(\tau)$}}
                        \put(70,9){\makebox(0,0)[cc]{$\alpha_+(\tau)$}}
\end{picture}
\caption{Case (3) - positive and negative zeros}
\end{center}
\end{figure}
\end{remark}

\begin{example}\label{drie}
Let  $\delta_1, \delta_2 >0$ and consider the function
\[
 Q(z)=z^2 \left( \frac{\delta_1}{-1-z}+\frac{\delta_2}{1-z} \right), \quad z \in \dC \setminus\{-1,1\}.
\]
Then $Q(z)$ is of the form \eqref{fack} with
\[
 R(z)=z^2, \quad M(z)=\frac{\delta_1}{-1-z}+\frac{\delta_2}{1-z},
\]
with $M(z) \in \mathbf{N}_0$. Hence $Q(z) \in \mathbf{N}_1$ with GZNT at $\alpha=0$ and GPNT at
$\beta=\infty$. The function $Q(z)$ is holomorphic in a neighborhood of $z=0$. Clearly, $Q(0)=0$, $Q'(0)=0$, and $Q''(0)=2(\delta_2-\delta_1)$. Hence,
\[
Q''(0)=0 \quad \Leftrightarrow \quad \delta_1=\delta_2,
\]
in which case $Q'''(0)=12 \delta_1=12 \delta_2>0$.  Therefore with regard to Theorem \ref{mainth}
one obtains
\begin{enumerate}
\item $\delta_2 > \delta_1$ implies Case 2a;
\item $\delta_2 < \delta_1$ implies Case 2b;
\item $\delta_2 = \delta_1$ implies Case 3.
\end{enumerate}
\end{example}

\section{Part of the path on the real line.}

Let $Q(z) \in \mathbf{N}_1$. The path of the GZNT may hit the real line coming from $\dC^+$ and immediately
return to $\dC^+$ as the example $Q(z)=z^3$ shows. However it is also possible that the path will have a part
of  the real line in common as the example $Q(z)=z^2$ shows. In this section it is assumed that  the GZNT
$\alpha$ belongs to the real line and the interest is in the existence of $\ep > 0$ such that
\[
[\alpha-\ep,\alpha]\subset\mathcal{F}_Q \quad \mbox{or} \quad
[\alpha,\alpha+\ep]\subset\mathcal{F}_Q.
\]
According to Theorem \ref{tauonreal} this implies that   $M(z)$ is holomorphic on
\[
(\alpha-\ep,\alpha) \quad \mbox{or} \quad (\alpha,\alpha+\ep),
\]
respectively. Recall that if the function $M(z) \in \mathbf{N}_0$ is holomorphic on an interval
$(\alpha,\beta)$, then $M(z)$ has (possibly improper) limits at the end points of the interval and  $M(\alpha+) \in [-\infty, \infty)$ and $M(\beta-) \in (-\infty, \infty]$. The fact that those limits are equal to the nontangential limits at $\alpha$ and $\beta$, respectively, will be extensively used in the proof below.
 
\begin{theorem}\label{realpath}
Let $Q(z)\in\mathbf{N}_1$ be of the form
\[
 Q(z)=(z-\alpha)^2M(z) \quad \mbox{or} \quad
 Q(z)=\frac{(z-\alpha)^2}{(z-\beta)(z-\bar{\beta})} M(z),
\]
with  $\alpha \in \dR$, $\beta \in \dC^+ \cup \dR$, $\beta \neq \alpha$, and $M(z) \in \mathbf{N}_0$. Then
the following statements are valid:
\begin{enumerate}\def\labelenumi{\rm (\roman{enumi})}
\item
There exists $\varepsilon>0$ such that $(\alpha-\ep,\alpha]\subset \mathcal{ F}_Q$ if and only if
$M(z)$ is holomorphic on $(\alpha-\gamma,\alpha)$ for some $\gamma >0$
and $M(\alpha-) \in (0,\infty]$. 
\item
There exists  $\ep>0$ such that $[\alpha,\alpha+\ep)\subset \mathcal{F}_Q$ if and only if 
$M(z)$ is holomorphic on $(\alpha,\alpha+\gamma)$ for some $\gamma>0$
 and $M(\alpha+) \in [-\infty, 0)$.
\end{enumerate}
\end{theorem}

\begin{proof}
Since  the proofs of the statements (i) and (ii)  are analogous, only (i) will be shown. By a translation the
statement can be easily reduced to the case $\alpha=0$.  Hence, from now on it is assumed that $M(z)$ is
holomorphic on some interval $(-\gamma,0)$, with $\gamma>0$. The proof will be carried out for the functions
\[
 z^2M(z) \quad \mbox{and} \quad \frac{z^2M(z)}{(z-\beta)(z-\bar{\beta})},
\]
respectively, in three steps. \\

\textit{Step 1.}  The function $R(z)=z^2M(z)$ is considered. Since $M(z) \in \mathbf{N}_0$ it has the
representation \eqref{nev'}.  Recall that
\begin{equation}\label{hash}
 \lim_{z \uparrow 0} -z M(z)=\sigma(\{0\}),
\end{equation}
cf. \eqref{nev++}. Note that if $\sigma(\{0\})=0$,  dominated convergence on $(-\infty,-\gamma)$  and monotone convergence on $(0, \infty)$ implies that
\begin{equation}\label{ooo}
M(0-)=a+\int_{\dR }   \frac{d\sig(s)}{s (s^2+1)} \in \dR \cup \{\infty\},
\end{equation}
cf. \eqref{rreff}.

In the general case $\sigma(\{0\}) \neq 0$
the function $R(z)= z^2M(z)$ can be written as
\[
 R(z)=az^2+bz^3 -z \sigma(\{0\})
 +z^2 \int_{\dR \setminus \{0\}} \left( \frac{1}{s-z}-\frac{s}{s^2+1} \right) d\sigma(s)
\]
and  it is holomorphic on $(-\gamma, 0)$. A straightforward calculation implies that
\begin{equation}\label{kuhstrich}
R'(z)=-\sig(\{0\})+2az+3bz^2+zT(z),
\end{equation}
where the function $T(z)$ is given by
$$
T(z)=\int_{\dR\setminus \{0\}} h(s,z) \,\frac{d\sig(s)}{s-z},
$$
with
$$
h(s,z)=\frac{2+2sz}{s^2+1}+\frac{z}{s-z}.
$$
It will be shown that
\begin{equation}\label{??}
\lim_{z\uparrow 0}zT(z)=0.
\end{equation}
In order to see this,  observe that
\begin{equation}\label{limzh}
 \lim_{z\uparrow 0} z \left( h(s,z) \frac{1}{s-z} \right) =0,\quad s\in\dR\setminus\{0\}.
\end{equation}
The argument will be finished via dominated convergence. For this purpose
write the function $T(z)$ as $T(z)=T_1(z)+T_2(z)$ with
$$
T_1(z)=\int_{(0,1/2)} h(s,z) \,\frac{d\sig(s)}{s-z},
\quad 
T_2(z)= \int_{\dR\setminus(-\gamma,1/2)} h(s,z) \,\frac{d\sig(s)}{s-z}.
$$
Note that for $s\in (0,1/2)$, $z\in(-1/2,0)$ the following
estimations hold:
\begin{equation}\label{est1}
 -1< \frac{z}{s-z}<0
\end{equation}
and
\begin{equation}\label{est2}
\frac{6}{5}< \frac{2-s}{s^2+1}<\frac{2+2sz}{s^2+1}<2.
\end{equation}
It follows from \eqref{est1} and \eqref{est2} that for $s \in (0,1/2)$ and $z \in (-1/2,0)$
one has
\begin{equation}\label{hschaetz}
 \frac{1}{5}<h(s,z)<2.
\end{equation}
The last estimate together with \eqref{limzh}  and \eqref{est1}
implies via dominated convergence that
\[
\lim_{z\uparrow 0}zT_1(z)=0.
\]
Note that  
for $s < -\gamma$ and $z\in (-\gamma/2,0)$ one has
\[
 |s-z| \ge |s|-\frac{\gamma}{2}, 
\]
and for $s \ge \frac{1}{2}$ and $z\in (-\gamma/2,0)$ one has
\[
 \quad |s-z| \ge s, 
\]
so that there are nonnegative constants $A,B$, such that
\[
 \left | \frac{1}{s-z} \right| \le \frac{A}{|s|}, \quad  \left | \frac{2+2sz}{s^2+1} \right| \le \frac{B}{|s|},
 \quad s\in\dR\setminus(-\gamma,1/2),\ z\in(-\gamma/2,0).
\]
Hence it follows that there exists $C\geq0$ such that 
\[
\left| h(s,z) \frac{1}{s-z} \right| \le \frac{C}{s^2},
\quad \quad s\in\dR\setminus(-\gamma,1/2),\ z\in(-\gamma/2,0).
\]
This estimate together with \eqref{limzh}   implies via dominated convergence that
\[
\lim_{z\uparrow 0}zT_2(z)=0.
\]
Consequently, \eqref{??} follows.

 Now consider the following three mutually exclusive cases, with regard to the 
 behavior of $\sigma$ in a neighborhood of zero. \\

\textit{Case (a)}. Assume that
\begin{equation}\label{1a}
 \sig(\{0\})>0.
\end{equation}
Then the identity \eqref{kuhstrich} together with \eqref{??} implies that there is some $\ep>0$
such that $R(z)$ is holomorphic with $R'(z)<0$ on $(-\ep,0)$. By Theorem \ref{mainth} the interval
$(-\ep,0)$ is contained in $\mathcal{F}_Q$.   Note that in this case it follows from \eqref{hash} that
\begin{equation}\label{1a'}
 M(0-)=+\infty.
\end{equation}

\textit{Case (b)}. Assume that
\begin{equation}\label{1b}
 \sig(\{0\})=0 \quad \mbox{and} \quad \int_{(0,1/2)} \frac{d\sig (s)}{s}=+\infty.
\end{equation}
By the estimate  \eqref{hschaetz} and monotone convergence it follows that
$$
\int_{(0,1/2)} h(s,z) \,\frac{d\sig(s)}{s-z} \ge \frac{1}{5} \int_{(0,1/2)}\frac{d\sig(s)}{s-z}\to +\infty,
 \quad z\uparrow 0.
$$
Consequently, $T_1(0-)=+\infty$. Since $T_2(z)$ is holomorphic on the interval $(-\gamma,1/2)$,
one obtains
$T(0-)=+\infty$. As $T$ is holomorphic on $(-\gamma,0)$, the identity \eqref{kuhstrich} implies
again that $R'(z)<0$ on some interval $(-\ep,0)$, which is therefore contained in $\mathcal{F}_Q$.
Note that in this case it follows from \eqref{ooo} that
\begin{equation}\label{1b'}
M(0-)=+\infty.
\end{equation}

\textit{Case (c)}. Assume that
\begin{equation}\label{1c}
\sig(\{0\})=0 \quad \mbox{and} \quad \int_{(0,1/2)} \frac{d\sig (s)}{s} < + \infty.
\end{equation}
 Then
\begin{equation}\label{altkuhstrich}
R'(z)=2z\left(a+\int_{\dR\setminus \{0\}} \frac{d\sig(s)}{s(s^2+1)}\right) + z^2S(z),
\end{equation}
with
\begin{equation}\label{bbb}
S(z)=3b+ \int_{\dR\setminus \{0\}}  \frac{3s-2z}{s(s-z)^2}  \, d\sig(s).
\end{equation}
It will be shown by dominated convergence that
\begin{equation}\label{limzS}
\lim_{z \uparrow 0} zS(z)=0.
\end{equation}
To do this, it suffices to find an integrable upper bound when the integrand in \eqref{bbb}
is multiplied by $z$.
For $s>0$ and $z<0$ one has $0< 3s-2z < 3(s-z)$, so that
\[
 0 \le \frac{3s-2z}{s(s-z)^2} \le  \frac{3}{s(s-z)}.
\]
Since  $|z| < s-z$ for $s >0$ and $z<0$, it follows that
\begin{equation}\label{hen1}
   \left| z \frac{3s-2z}{s(s-z)^2} \right| \le \frac{3}{s}, \quad s>0, \quad z<0,
\end{equation}
and in particular  this estimate holds for $0<s<1/2$ and $z <0$.
Now let $s \ge 1/2$
and assume that $-\gamma/2 < z<0$. Then $|s-z|\ge |s|$ and thus
\begin{equation}\label{hen5}
  \left| \frac{3s-2z}{s(s-z)^2} \right| \le 
  \frac{3|s|+|\gamma|}{|s|^3}.
\end{equation}
Next let $s \le -\gamma$ and assume that $-\gamma/2 < z<0$. Then
\[
 |s-z| \ge |s|-|\gamma|/2,
\]
and thus
\begin{equation}\label{hen4}
  \left| \frac{3s-2z}{s(s-z)^2} \right| \le 
  \frac{3|s|+|\gamma|}{|s|(|s|-|\gamma|/2)^2}.
\end{equation}
Hence, combining \eqref{hen5} and \eqref{hen4}, there exists $A \ge 0$ such that
\begin{equation}\label{hen2}
 \left| z \frac{3s-2z}{s(s-z)^2} \right| \le  \frac{A}{|s|^2}, \quad -\gamma/2 < z<0, \quad
 s \ge 1/2 \,\,\, \mbox{or} \,\,\, s \le -\gamma.
\end{equation}
The estimates \eqref{hen1} and \eqref{hen2}, together with the integrability condition \eqref{int'} and
the integrability assumption in \eqref{1c}, show that \eqref{limzS} is satisfied.
In a similar way, it can be shown that
\begin{equation}\label{r1lim}
\lim_{z\uparrow 0} S(z)=3b+3\int_{\dR\setminus \{0\}}\frac{d\sig(s)}{s^2}\in[0,+\infty].
\end{equation}
In fact, this can be seen via dominated convergence on the negative axis (use the estimate
\eqref{hen4}) and by monotone convergence on the positive axis. 

Observe that in the present case it follows from \eqref{ooo} that $M(0-) \in \dR$.
There are three situations to consider:

\begin{itemize}

\item
$M(0-)>0$. Then the identities \eqref{altkuhstrich} and \eqref{limzS} imply that
$R'(z)<0$ on  some interval $(-\ep,0)$, which is therefore contained in $\mathcal{F}_Q$.

\item
$M(0-)< 0$. Similarly, one obtains $R'(z)>0$ on  some interval $(-\ep,0)$, which is
therefore not contained in $\mathcal{F}_Q$.

\item
$M(0-)=0$.
It follows from \eqref{ooo} and \eqref{altkuhstrich}
that $R'(z)=z^2S(z)$.  The identity \eqref{r1lim} implies that $\lim_{z\uparrow 0} S(z)\in(0,+\infty]$
(otherwise $b$ and $\sigma$ and by \eqref{ooo} also $a$ would vanish, so that $M(z) \equiv 0$,
contradicting the fact that $R(z)\in\mathbf{N}_1$).
Hence, $R'(z)>0$ on some interval $(-\ep,0)$, which is therefore not contained in $\mathcal{F}_Q$.  \\
\end{itemize}

\textit{Step 2.}  In order to treat the general case it will now be proved that
\begin{equation}\label{rprim}
 \lim_{z \uparrow 0} \frac{R'(z)}{R(z)} = -\infty,
\end{equation}
in each of the cases (a), (b), and (c) in Step 1. \\

\textit{Case (a)}.
Recall that  $R'(z) <0$ on some interval $(-\varepsilon,0)$ with $\lim_{z \uparrow 0} R'(z)=-\sig(\{0\})$ according to \eqref{kuhstrich}.
Since $R(z) >0$ on some interval $(-\varepsilon_1,0)$ by \eqref{1a'}  with
$\lim_{z \uparrow 0} R(z)=0$ (see \eqref{hash}), it
follows that \eqref{rprim} holds. \\

\textit{Case (b)}.
Recall that $T(0-) =+ \infty$, so that
\eqref{kuhstrich} implies that
\[
\lim_{z \uparrow 0}\frac{R'(z)}{z} \to +\infty.
\]
Furthermore
\[
\lim_{z \uparrow 0} \frac{R(z)}{z} = \lim_{z \uparrow 0} zM(z) =0,
\]
and note that $z M(z) \uparrow 0$ for $z \uparrow 0$ (see \eqref{1b'}).
Thus \eqref{rprim} follows. \\

\textit{Case (c)}.
It follows from \eqref{ooo}, \eqref{altkuhstrich}, and
\eqref{limzS} that
\[
\lim_{z \uparrow 0} \frac{R'(z)}{z}=2M(0-).
\]

If $M(0-) \neq 0$, then
\begin{equation}
 \lim_{z \uparrow 0} \frac{z R'(z)}{R(z)} =
  \lim_{z \uparrow 0} \frac{R'(z)}{z} \frac{1}{M(z)} =2,
\end{equation}
from which \eqref{rprim} follows.

If $M(0-)=0$, then $M(z) \uparrow0$ as $z  \uparrow 0$, so that
\begin{equation}
 \lim_{z \uparrow 0} \frac{R'(z)}{R(z)} =  \lim_{z \uparrow 0} \frac{S(z)}{M(z)}=-\infty,
\end{equation}
and, again, \eqref{rprim} follows. \\

\textit{Step 3.}  Consider the function $Q(z)$ defined by
\[
 Q(z)=\frac{z^2M(z)}{(z-\beta)(z-\bar{\beta})}=\frac{R(z)}{D(z)},
\]
where $R(z)=z^2M(z)$ as in Step 1 and $D(z)=(z-\beta)(z-\bar{\beta})$. Then
\begin{equation}\label{deriv}
 Q'(z)=\frac{R'(z)D(z)-R(z)D'(z)}{D(z)^2}=M(z)\,\frac{z^2}{D(z)} \left( \frac{R'(z)}{R(z)}-\frac{D'(z)}{D(z)} \right).
\end{equation}
Note that $D(z) >0$ for $z \in \dR$ and $z \neq \beta$, and that
\[
 \frac{D'(0)}{D(0)}=- \frac{\RE \beta}{|\beta|^2}.
\]
The identity \eqref{deriv}  and the limit in \eqref{rprim} now imply that
\[
Q'(z) <0,\,\, z \in (-\varepsilon,0), \mbox{ for some $\varepsilon > 0$}
\quad \Leftrightarrow \quad 0< M(0-) \le \infty,
\]
and
\[
Q'(z) \ge 0, \,\, z \in (-\varepsilon,0), \mbox{ for some $\varepsilon > 0$} \quad \Leftrightarrow \quad M(0-) \le  0.
\]
This completes the proof of the theorem.
\end{proof}

\begin{example}
If $Q(z)=z^{2+\rho}$, $0 < \rho<1$, then  $Q(z)=z^2M(z)$ with $M(z)=z^\rho \in \mathbf{N}_0$.  Hence $Q(z)
\in \mathbf{N}_1$ and $z=0$  
is the GZNT of $Q(z)$.  With the interpretation of $z^\rho$
as in the introduction, it follows that $M(z)$ is holomorphic on
$(0,+\infty)$. It is not
difficult to see that $\lim_{z \downarrow 0} M(z)=0$. Hence, the conditions of Theorem \ref{realpath} are not
satisfied. The path is indicated in Figure 2.
\end{example}

\begin{example}
If $Q(z)=z^2 \log z$, then $Q(z)=z^2 M(z)$ with $M(z) =\log z$. Hence $Q(z) \in \mathbf{N}_1$ and $z=0$ is the GZNT of $Q(z)$. Clearly $\lim_{z \downarrow 0} \log z=-\infty$, so that by Theorem
\ref{realpath}, there exists $\gamma >0$ such that $(0, \gamma) \subset \mathcal{F}_Q$. It can be shown that
$\gamma=1/\sqrt{e}$ is the maximal possible value.
\end{example}

In the special case that $Q(z) \in \mathbf{N}_1$ is holomorphic in a neighborhood
of the GZNT $\alpha \in \dR$,  a combination of Theorem \ref{realpath} and
Proposition \ref{zeros>0} leads to the following description, which
agrees with Theorem \ref{mainth}.

\begin{corollary}\label{zeros>0+}
Let $Q(z) \in \mathbf{N}_1$ be holomorphic in a neighborhood
of the GZNT $\alpha \in \dR$.  Then the following cases appear:
\begin{itemize}
\item[(1)]   $Q'(\alpha) < 0$: There exists $\varepsilon>0$ such that $(\alpha-\ep, \alpha+\ep)\subset \mathcal{ F}_Q$.

\item[(2a)]  $Q'(\alpha) = 0$, $Q''(\alpha) > 0$: There exists $\varepsilon>0$ such that $(\alpha-\ep,\alpha]\subset \mathcal{ F}_Q$.

 \item[(2b)] $Q'(\alpha) = 0$, $Q''(\alpha) < 0$: There exists $\varepsilon>0$ such that $[\alpha,\alpha+\ep ) \subset \mathcal{ F}_Q$.

\item[(3)]  $Q'(\alpha) = 0$, $Q''(\alpha) = 0$: There exists no $\varepsilon>0$ such that $(\alpha-\ep,\alpha]\subset \mathcal{ F}_Q$ or $[\alpha,\alpha+\ep ) \subset \mathcal{ F}_Q$.
\end{itemize}
\end{corollary}

In order to show how this follows from Theorem \ref{realpath}
assume for simplicity that $Q(z)$ is of the form $Q(z)=(z-\alpha)^2M(z)$ with $M(z) \in \mathbf{N}_0$.
Then the condition that $Q(z)$ is holomorphic around $\alpha$ implies that
$M(z)$ has an isolated first order pole at $\alpha$ or that $M(z)$ is holomorphic
around $\alpha$.

If $M(z)$ has a first order pole at $\alpha$, then
\begin{equation}\label{pol}
 M(z)=\frac{c}{z-\alpha} + \varphi(z),
\end{equation}
with $c<0$ and $\varphi(z)$ holomorphic around $\alpha$. Hence, in this case
\[
 Q'(\alpha) =c <0.
\]

If $M(z)$ is holomorphic at $\alpha$, then
\[
Q'(\alpha)=0, \quad Q''(\alpha)=2 M(\alpha), \quad Q'''(\alpha)=6M'(\alpha).
\]

If Case (1) prevails, then $Q'(\alpha)<0$. Hence the function $M(z)$ must be of the form \eqref{pol}.
(It also follows from \eqref{pol} and Theorem \ref{realpath} that the assertion in (1) holds).

If Cases (2a), (2b), or (3) prevail, then $Q'(\alpha)=0$. Hence the function $M(z)$ must be holomorphic at $\alpha$.
In particular in Case (2a) $M(\alpha)>0$ and in Case (2b) $M(\alpha)<0$, so that assertions follows from Theorem \ref{realpath}. Finally, in Case (3) $M(\alpha)=0$ and the assertion in (3) follows from Theorem \ref{realpath}.

\begin{remark}\label{53}
It may be helpful to reconsider some of the  concrete simple examples given earlier
in terms of Theorem \ref{realpath} and Corollary \ref{zeros>0}.

If $Q(z)=-z$, then $Q'(0)=-1$.  Since $Q(z)=z^2M(z)$ with $M(z)=-1/z \in \mathbf{N}_0$,
one has  $M(0+) =-\infty$, $M(0-)=\infty$.
 
If $Q(z)=z^2$, then $Q'(0)=0$, $Q''(0)=2$. Since
$Q(z)=z^2M(z)$ with $M(z)=1 \in \mathbf{N}_0$, one has  $M(0)=1$.
 
If $Q(z)=z^3$, then  $Q''(0)=0$, and $Q'''(0)=6$. Since
$Q(z)=z^2M(z)$ with $M(z)=z \in \mathbf{N}_0$, one has $M(0)=0$
 
If $Q(z)$ is as in Example \ref{drie}, then
$Q'(0)=0$, $Q''(0)=2(\delta_2-\delta_1)$, and $Q'''(0)=2(\delta_1+\delta_2)$, and one has
$M(0)= \delta_2-\delta_1$.
\end{remark}

\section{The path associated with a special function}\label{twe}

Each of the factors
in \eqref{einz} belongs to $\mathbf{N}_1$ and has 
$\alpha \in \dC^+ \cup \dR \cup{\infty}$ as GZNT and $\beta
\in \dC^+ \cup \dR \cup{\infty}$ as GPNT. 
In this section the path $\alpha(\tau)$ of the GZNT for this factor
will be described.  The extreme case
with $\beta=\infty$ can be treated as in the introduction (after a shift). The other extreme case with
$\alpha=\infty$ can be treated similarly.
Therefore it suffices to consider the function $Q(z)$ defined by
\begin{equation}
\label{einz-}
 Q(z)=\frac{(z-\alpha)(z-\overline{\alpha})}{(z-\beta)(z-\overline{\beta })},
 \quad \alpha, \beta \in \dC^+ \cup \dR.
\end{equation}

Observe that the equation $\IM Q(z)=0$ where $z=x+ \ii y$, is equivalent to either the equation
\begin{equation}\label{EINS}
y=0,
\end{equation}
or the equation
\begin{equation}\label{ZWEI}
 (\RE \alpha-\RE \beta) (x^2+y^2)+(|\beta|^2-|\alpha|^2)x+|\alpha|^2 \RE \beta-|\beta|^2 \RE \alpha=0.
\end{equation}
If $\RE \alpha \ne \RE \beta$, then the equation \eqref{ZWEI}  is equivalent to 
\[
\left( x-\frac{|\be|^2-|\al|^2}{2\Re(\beta-\alpha)}\right)^2+y^2= \left(
\Re\alpha-\frac{|\beta|^2-|\alpha|^2}{2\Re(\beta-\alpha)}\right)^2 +\left(\Im\alpha\right)^2.
\]
This defines a circle $C$ which can be also described by the condition that it contains $\al$ and $\be$ and
that its center lies on $\dR$. If $\RE \alpha = \RE \beta$, then the equation \eqref{ZWEI}  is equivalent to
\[
x=\RE \alpha,
\]
which defines a vertical line $C$ through $x =\RE \alpha$.
Note that by Theorem \ref{charact} and Corollary \ref{quh}:
\[
\mathcal{ F}_Q \subset (C\cap \dC^+)\cup\dR\cup\{\infty\}.
\]
It is straightforward to check that the sign of $Q'(x)$, $x \in \dR$ (with the possible exception of $\beta$
when $\beta \in \dR$), is given by the sign of the polynomial
\begin{equation}\label{DREI}
  (\RE \alpha-\RE \beta) x^2+(|\beta|^2-|\alpha|^2)x+|\alpha|^2 \RE \beta-|\beta|^2 \RE \alpha.
\end{equation}

Denote for $\RE \alpha \neq \RE \beta$ the intersections of $C$ with $\dR$ by $P_l$ and $P_r$ in such way that $P_l <P_r$. Denote for $\RE \alpha=\RE \beta$ the intersection of $C$ with $\dR$ by $P$.
There are now three cases to consider.

\textit{Case 1}. $\RE \alpha > \RE \beta$. Then it follows from \eqref{ZWEI} and \eqref{DREI} that
$Q'(x)<0$ for $x\in(P_l,P_r)$ and $Q'(x)>0$ for $x\in\dR\setminus(P_l,P_r)$.
Hence,
 \[
 \mathcal{F}_Q=(C\cap \dC^+)\cup[P_l,P_r].
\]
The direction of the path is indicated in Figure 8.

\begin{figure}[hbt]\label{circle1}
\begin{center}
\begin{picture}(95,32)(0,13)
\linethickness{0.05mm} \put(0,20){\line(1,0){90}} \thicklines
 \put(65,20){\vector(-1,0){10}}
 \put(55,20){\line(-1,0){30}}
 \multiput(64.99,20.5)(0.01,-0.5){1}{\line(0,-1){0.5}}
\multiput(64.98,21)(0.02,-0.5){1}{\line(0,-1){0.5}}
 \multiput(64.94,21.49)(0.03,-0.5){1}{\line(0,-1){0.5}}
\multiput(64.9,21.99)(0.04,-0.5){1}{\line(0,-1){0.5}}
 \multiput(64.84,22.49)(0.06,-0.5){1}{\line(0,-1){0.5}}
\multiput(64.78,22.98)(0.07,-0.49){1}{\line(0,-1){0.49}}
\multiput(64.7,23.47)(0.08,-0.49){1}{\line(0,-1){0.49}}
\multiput(64.6,23.96)(0.09,-0.49){1}{\line(0,-1){0.49}}
\multiput(64.5,24.45)(0.1,-0.49){1}{\line(0,-1){0.49}}
\multiput(64.38,24.94)(0.12,-0.48){1}{\line(0,-1){0.48}}
\multiput(64.25,25.42)(0.13,-0.48){1}{\line(0,-1){0.48}}
\multiput(64.11,25.9)(0.14,-0.48){1}{\line(0,-1){0.48}}
\multiput(63.96,26.37)(0.15,-0.47){1}{\line(0,-1){0.47}}
\multiput(63.79,26.84)(0.16,-0.47){1}{\line(0,-1){0.47}}
\multiput(63.62,27.31)(0.18,-0.47){1}{\line(0,-1){0.47}}
\multiput(63.43,27.77)(0.09,-0.23){2}{\line(0,-1){0.23}}
\multiput(63.23,28.23)(0.1,-0.23){2}{\line(0,-1){0.23}}
\multiput(63.02,28.68)(0.11,-0.23){2}{\line(0,-1){0.23}}
\multiput(62.8,29.12)(0.11,-0.22){2}{\line(0,-1){0.22}}
\multiput(62.56,29.57)(0.12,-0.22){2}{\line(0,-1){0.22}}
\multiput(62.32,30)(0.12,-0.22){2}{\line(0,-1){0.22}}
\multiput(62.07,30.43)(0.13,-0.21){2}{\line(0,-1){0.21}}
\multiput(61.8,30.85)(0.13,-0.21){2}{\line(0,-1){0.21}}
\multiput(61.52,31.27)(0.14,-0.21){2}{\line(0,-1){0.21}}
\multiput(61.24,31.67)(0.14,-0.2){2}{\line(0,-1){0.2}}
 \multiput(60.94,32.08)(0.15,-0.2){2}{\line(0,-1){0.2}}
\multiput(60.64,32.47)(0.1,-0.13){3}{\line(0,-1){0.13}}
\multiput(60.32,32.86)(0.11,-0.13){3}{\line(0,-1){0.13}}
\multiput(60,33.23)(0.11,-0.13){3}{\line(0,-1){0.13}}
 \multiput(59.66,33.6)(0.11,-0.12){3}{\line(0,-1){0.12}}
\multiput(59.32,33.96)(0.11,-0.12){3}{\line(0,-1){0.12}}
\multiput(58.96,34.32)(0.12,-0.12){3}{\line(1,0){0.12}}
\multiput(58.6,34.66)(0.12,-0.11){3}{\line(1,0){0.12}}
 \multiput(58.23,35)(0.12,-0.11){3}{\line(1,0){0.12}}
\multiput(57.86,35.32)(0.13,-0.11){3}{\line(1,0){0.13}}
\multiput(57.47,35.64)(0.13,-0.11){3}{\line(1,0){0.13}}
\multiput(57.08,35.94)(0.13,-0.1){3}{\line(1,0){0.13}}
 \multiput(56.67,36.24)(0.2,-0.15){2}{\line(1,0){0.2}}
\multiput(56.27,36.52)(0.2,-0.14){2}{\line(1,0){0.2}}
 \multiput(55.85,36.8)(0.21,-0.14){2}{\line(1,0){0.21}}
\multiput(55.43,37.07)(0.21,-0.13){2}{\line(1,0){0.21}}
 \multiput(55,37.32)(0.21,-0.13){2}{\line(1,0){0.21}}
\multiput(54.57,37.56)(0.22,-0.12){2}{\line(1,0){0.22}}
\multiput(54.12,37.8)(0.22,-0.12){2}{\line(1,0){0.22}}
\multiput(53.68,38.02)(0.22,-0.11){2}{\line(1,0){0.22}}
\multiput(53.23,38.23)(0.23,-0.11){2}{\line(1,0){0.23}}
\multiput(52.77,38.43)(0.23,-0.1){2}{\line(1,0){0.23}}
\multiput(52.31,38.62)(0.23,-0.09){2}{\line(1,0){0.23}}
\multiput(51.84,38.79)(0.47,-0.18){1}{\line(1,0){0.47}}
\multiput(51.37,38.96)(0.47,-0.16){1}{\line(1,0){0.47}}
\multiput(50.9,39.11)(0.47,-0.15){1}{\line(1,0){0.47}}
\multiput(50.42,39.25)(0.48,-0.14){1}{\line(1,0){0.48}}
\multiput(49.94,39.38)(0.48,-0.13){1}{\line(1,0){0.48}}
\multiput(49.45,39.5)(0.48,-0.12){1}{\line(1,0){0.48}}
 \multiput(48.96,39.6)(0.49,-0.1){1}{\line(1,0){0.49}}
\multiput(48.47,39.7)(0.49,-0.09){1}{\line(1,0){0.49}}
\multiput(47.98,39.78)(0.49,-0.08){1}{\line(1,0){0.49}}
\multiput(47.49,39.84)(0.49,-0.07){1}{\line(1,0){0.49}}
 \multiput(46.99,39.9)(0.5,-0.06){1}{\line(1,0){0.5}}
\multiput(46.49,39.94)(0.5,-0.04){1}{\line(1,0){0.5}}
 \multiput(46,39.98)(0.5,-0.03){1}{\line(1,0){0.5}}
\multiput(45.5,39.99)(0.5,-0.02){1}{\vector(1,0){0.5}}
 \multiput(45,40)(0.5,-0.01){1}{\line(1,0){0.5}}
\multiput(44.5,39.99)(0.5,0.01){1}{\line(1,0){0.5}}
 \multiput(44,39.98)(0.5,0.02){1}{\line(1,0){0.5}}
\multiput(43.51,39.94)(0.5,0.03){1}{\line(1,0){0.5}}
 \multiput(43.01,39.9)(0.5,0.04){1}{\line(1,0){0.5}}
\multiput(42.51,39.84)(0.5,0.06){1}{\line(1,0){0.5}}
 \multiput(42.02,39.78)(0.49,0.07){1}{\line(1,0){0.49}}
\multiput(41.53,39.7)(0.49,0.08){1}{\line(1,0){0.49}}
 \multiput(41.04,39.6)(0.49,0.09){1}{\line(1,0){0.49}}
\multiput(40.55,39.5)(0.49,0.1){1}{\line(1,0){0.49}}
 \multiput(40.06,39.38)(0.48,0.12){1}{\line(1,0){0.48}}
\multiput(39.58,39.25)(0.48,0.13){1}{\line(1,0){0.48}}
 \multiput(39.1,39.11)(0.48,0.14){1}{\line(1,0){0.48}}
\multiput(38.63,38.96)(0.47,0.15){1}{\line(1,0){0.47}}
 \multiput(38.16,38.79)(0.47,0.16){1}{\line(1,0){0.47}}
\multiput(37.69,38.62)(0.47,0.18){1}{\line(1,0){0.47}}
 \multiput(37.23,38.43)(0.23,0.09){2}{\line(1,0){0.23}}
\multiput(36.77,38.23)(0.23,0.1){2}{\line(1,0){0.23}}
 \multiput(36.32,38.02)(0.23,0.11){2}{\line(1,0){0.23}}
\multiput(35.88,37.8)(0.22,0.11){2}{\line(1,0){0.22}}
 \multiput(35.43,37.56)(0.22,0.12){2}{\line(1,0){0.22}}
\multiput(35,37.32)(0.22,0.12){2}{\line(1,0){0.22}}
 \multiput(34.57,37.07)(0.21,0.13){2}{\line(1,0){0.21}}
\multiput(34.15,36.8)(0.21,0.13){2}{\line(1,0){0.21}}
 \multiput(33.73,36.52)(0.21,0.14){2}{\line(1,0){0.21}}
\multiput(33.33,36.24)(0.2,0.14){2}{\line(1,0){0.2}}
 \multiput(32.92,35.94)(0.2,0.15){2}{\line(1,0){0.2}}
\multiput(32.53,35.64)(0.13,0.1){3}{\line(1,0){0.13}}
 \multiput(32.14,35.32)(0.13,0.11){3}{\line(1,0){0.13}}
\multiput(31.77,35)(0.13,0.11){3}{\line(1,0){0.13}}
 \multiput(31.4,34.66)(0.12,0.11){3}{\line(1,0){0.12}}
\multiput(31.04,34.32)(0.12,0.11){3}{\line(1,0){0.12}}
 \multiput(30.68,33.96)(0.12,0.12){3}{\line(1,0){0.12}}
  \multiput(30.34,33.6)(0.11,0.12){3}{\line(0,1){0.12}}
 \multiput(30,33.23)(0.11,0.12){3}{\line(0,1){0.12}}
\multiput(29.68,32.86)(0.11,0.13){3}{\line(0,1){0.13}}
 \multiput(29.36,32.47)(0.11,0.13){3}{\line(0,1){0.13}}
\multiput(29.06,32.08)(0.1,0.13){3}{\line(0,1){0.13}}
 \multiput(28.76,31.67)(0.15,0.2){2}{\line(0,1){0.2}}
\multiput(28.48,31.27)(0.14,0.2){2}{\line(0,1){0.2}}
 \multiput(28.2,30.85)(0.14,0.21){2}{\line(0,1){0.21}}
\multiput(27.93,30.43)(0.13,0.21){2}{\line(0,1){0.21}}
 \multiput(27.68,30)(0.13,0.21){2}{\line(0,1){0.21}}
\multiput(27.44,29.57)(0.12,0.22){2}{\line(0,1){0.22}}
 \multiput(27.2,29.12)(0.12,0.22){2}{\line(0,1){0.22}}
\multiput(26.98,28.68)(0.11,0.22){2}{\line(0,1){0.22}}
 \multiput(26.77,28.23)(0.11,0.23){2}{\line(0,1){0.23}}
\multiput(26.57,27.77)(0.1,0.23){2}{\line(0,1){0.23}}
 \multiput(26.38,27.31)(0.09,0.23){2}{\line(0,1){0.23}}
\multiput(26.21,26.84)(0.18,0.47){1}{\line(0,1){0.47}}
 \multiput(26.04,26.37)(0.16,0.47){1}{\line(0,1){0.47}}
\multiput(25.89,25.9)(0.15,0.47){1}{\line(0,1){0.47}}
 \multiput(25.75,25.42)(0.14,0.48){1}{\line(0,1){0.48}}
\multiput(25.62,24.94)(0.13,0.48){1}{\line(0,1){0.48}}
 \multiput(25.5,24.45)(0.12,0.48){1}{\line(0,1){0.48}}
\multiput(25.4,23.96)(0.1,0.49){1}{\line(0,1){0.49}}
 \multiput(25.3,23.47)(0.09,0.49){1}{\line(0,1){0.49}}
\multiput(25.22,22.98)(0.08,0.49){1}{\line(0,1){0.49}}
 \multiput(25.16,22.49)(0.07,0.49){1}{\line(0,1){0.49}}
\multiput(25.1,21.99)(0.06,0.5){1}{\line(0,1){0.5}}
 \multiput(25.06,21.49)(0.04,0.5){1}{\line(0,1){0.5}}
\multiput(25.02,21)(0.03,0.5){1}{\line(0,1){0.5}}
 \multiput(25.01,20.5)(0.02,0.5){1}{\line(0,1){0.5}}
\multiput(25,20)(0.01,0.5){1}{\line(0,1){0.5}}
 \thinlines \put(35,36){\line(0,1){2}}
\put(61,31){\line(0,1){2}}
 \put(36,35){$\beta$}
  \put(58,30){$\alpha$}
\put(25,19){\line(0,1){2}}
 \put(65,19){\line(0,1){2}}
 \put(25,17){\makebox(0,0)[cc]{$P_l$}}
\put(65,17){\makebox(0,0)[cc]{$P_r$}}

\end{picture}
\caption{The path of $\alpha(\tau)$ for $\RE \alpha > \RE \beta$.}
\end{center}
\end{figure}

\textit{Case 2}. $\RE \alpha < \RE \beta$. Then it follows from \eqref{ZWEI} and \eqref{DREI} that
$Q'(x)>0$ for $x\in(P_l,P_r)$ and $Q'(x)<0$ for $x\in\dR\setminus(P_l,P_r)$.
Hence,
 \[
 \mathcal{F}_Q=(C\cap \dC^+)\cup(-\infty, P_l] \cup [P_r, \infty).
\]
The direction of the path is indicated in Figure 9.

\begin{figure}[hbt]\label{circle2}
\begin{center}
\begin{picture}(90,32)(0,13)
\linethickness{0.05mm} \put(0,20){\line(1,0){90}} \thicklines
 \put(25,20){\vector(-1,0){10}}
 \put(15,20){\line(-1,0){15}}
 \put(90,20){\vector(-1,0){10}}
 \put(80,20){\line(-1,0){15}}
 \multiput(64.99,20.5)(0.01,-0.5){1}{\line(0,-1){0.5}}
\multiput(64.98,21)(0.02,-0.5){1}{\line(0,-1){0.5}}
 \multiput(64.94,21.49)(0.03,-0.5){1}{\line(0,-1){0.5}}
\multiput(64.9,21.99)(0.04,-0.5){1}{\line(0,-1){0.5}}
 \multiput(64.84,22.49)(0.06,-0.5){1}{\line(0,-1){0.5}}
\multiput(64.78,22.98)(0.07,-0.49){1}{\line(0,-1){0.49}}
\multiput(64.7,23.47)(0.08,-0.49){1}{\line(0,-1){0.49}}
\multiput(64.6,23.96)(0.09,-0.49){1}{\line(0,-1){0.49}}
\multiput(64.5,24.45)(0.1,-0.49){1}{\line(0,-1){0.49}}
\multiput(64.38,24.94)(0.12,-0.48){1}{\line(0,-1){0.48}}
\multiput(64.25,25.42)(0.13,-0.48){1}{\line(0,-1){0.48}}
\multiput(64.11,25.9)(0.14,-0.48){1}{\line(0,-1){0.48}}
\multiput(63.96,26.37)(0.15,-0.47){1}{\line(0,-1){0.47}}
\multiput(63.79,26.84)(0.16,-0.47){1}{\line(0,-1){0.47}}
\multiput(63.62,27.31)(0.18,-0.47){1}{\line(0,-1){0.47}}
\multiput(63.43,27.77)(0.09,-0.23){2}{\line(0,-1){0.23}}
\multiput(63.23,28.23)(0.1,-0.23){2}{\line(0,-1){0.23}}
\multiput(63.02,28.68)(0.11,-0.23){2}{\line(0,-1){0.23}}
\multiput(62.8,29.12)(0.11,-0.22){2}{\line(0,-1){0.22}}
\multiput(62.56,29.57)(0.12,-0.22){2}{\line(0,-1){0.22}}
\multiput(62.32,30)(0.12,-0.22){2}{\line(0,-1){0.22}}
\multiput(62.07,30.43)(0.13,-0.21){2}{\line(0,-1){0.21}}
\multiput(61.8,30.85)(0.13,-0.21){2}{\line(0,-1){0.21}}
\multiput(61.52,31.27)(0.14,-0.21){2}{\line(0,-1){0.21}}
\multiput(61.24,31.67)(0.14,-0.2){2}{\line(0,-1){0.2}}
 \multiput(60.94,32.08)(0.15,-0.2){2}{\line(0,-1){0.2}}
\multiput(60.64,32.47)(0.1,-0.13){3}{\line(0,-1){0.13}}
\multiput(60.32,32.86)(0.11,-0.13){3}{\line(0,-1){0.13}}
\multiput(60,33.23)(0.11,-0.13){3}{\line(0,-1){0.13}}
 \multiput(59.66,33.6)(0.11,-0.12){3}{\line(0,-1){0.12}}
\multiput(59.32,33.96)(0.11,-0.12){3}{\line(0,-1){0.12}}
\multiput(58.96,34.32)(0.12,-0.12){3}{\line(1,0){0.12}}
\multiput(58.6,34.66)(0.12,-0.11){3}{\line(1,0){0.12}}
 \multiput(58.23,35)(0.12,-0.11){3}{\line(1,0){0.12}}
\multiput(57.86,35.32)(0.13,-0.11){3}{\line(1,0){0.13}}
\multiput(57.47,35.64)(0.13,-0.11){3}{\line(1,0){0.13}}
\multiput(57.08,35.94)(0.13,-0.1){3}{\line(1,0){0.13}}
 \multiput(56.67,36.24)(0.2,-0.15){2}{\line(1,0){0.2}}
\multiput(56.27,36.52)(0.2,-0.14){2}{\line(1,0){0.2}}
 \multiput(55.85,36.8)(0.21,-0.14){2}{\line(1,0){0.21}}
\multiput(55.43,37.07)(0.21,-0.13){2}{\line(1,0){0.21}}
 \multiput(55,37.32)(0.21,-0.13){2}{\line(1,0){0.21}}
\multiput(54.57,37.56)(0.22,-0.12){2}{\line(1,0){0.22}}
\multiput(54.12,37.8)(0.22,-0.12){2}{\line(1,0){0.22}}
\multiput(53.68,38.02)(0.22,-0.11){2}{\line(1,0){0.22}}
\multiput(53.23,38.23)(0.23,-0.11){2}{\line(1,0){0.23}}
\multiput(52.77,38.43)(0.23,-0.1){2}{\line(1,0){0.23}}
\multiput(52.31,38.62)(0.23,-0.09){2}{\line(1,0){0.23}}
\multiput(51.84,38.79)(0.47,-0.18){1}{\line(1,0){0.47}}
\multiput(51.37,38.96)(0.47,-0.16){1}{\line(1,0){0.47}}
\multiput(50.9,39.11)(0.47,-0.15){1}{\line(1,0){0.47}}
\multiput(50.42,39.25)(0.48,-0.14){1}{\line(1,0){0.48}}
\multiput(49.94,39.38)(0.48,-0.13){1}{\line(1,0){0.48}}
\multiput(49.45,39.5)(0.48,-0.12){1}{\line(1,0){0.48}}
 \multiput(48.96,39.6)(0.49,-0.1){1}{\line(1,0){0.49}}
\multiput(48.47,39.7)(0.49,-0.09){1}{\line(1,0){0.49}}
\multiput(47.98,39.78)(0.49,-0.08){1}{\line(1,0){0.49}}
\multiput(47.49,39.84)(0.49,-0.07){1}{\line(1,0){0.49}}
 \multiput(46.99,39.9)(0.5,-0.06){1}{\line(1,0){0.5}}
\multiput(46.49,39.94)(0.5,-0.04){1}{\line(1,0){0.5}}
 \multiput(46,39.98)(0.5,-0.03){1}{\line(1,0){0.5}}
\multiput(45.5,39.99)(0.5,-0.02){1}{\vector(-1,0){0.5}}
 \multiput(45,40)(0.5,-0.01){1}{\line(1,0){0.5}}
\multiput(44.5,39.99)(0.5,0.01){1}{\line(1,0){0.5}}
 \multiput(44,39.98)(0.5,0.02){1}{\line(1,0){0.5}}
\multiput(43.51,39.94)(0.5,0.03){1}{\line(1,0){0.5}}
 \multiput(43.01,39.9)(0.5,0.04){1}{\line(1,0){0.5}}
\multiput(42.51,39.84)(0.5,0.06){1}{\line(1,0){0.5}}
 \multiput(42.02,39.78)(0.49,0.07){1}{\line(1,0){0.49}}
\multiput(41.53,39.7)(0.49,0.08){1}{\line(1,0){0.49}}
 \multiput(41.04,39.6)(0.49,0.09){1}{\line(1,0){0.49}}
\multiput(40.55,39.5)(0.49,0.1){1}{\line(1,0){0.49}}
 \multiput(40.06,39.38)(0.48,0.12){1}{\line(1,0){0.48}}
\multiput(39.58,39.25)(0.48,0.13){1}{\line(1,0){0.48}}
 \multiput(39.1,39.11)(0.48,0.14){1}{\line(1,0){0.48}}
\multiput(38.63,38.96)(0.47,0.15){1}{\line(1,0){0.47}}
 \multiput(38.16,38.79)(0.47,0.16){1}{\line(1,0){0.47}}
\multiput(37.69,38.62)(0.47,0.18){1}{\line(1,0){0.47}}
 \multiput(37.23,38.43)(0.23,0.09){2}{\line(1,0){0.23}}
\multiput(36.77,38.23)(0.23,0.1){2}{\line(1,0){0.23}}
 \multiput(36.32,38.02)(0.23,0.11){2}{\line(1,0){0.23}}
\multiput(35.88,37.8)(0.22,0.11){2}{\line(1,0){0.22}}
 \multiput(35.43,37.56)(0.22,0.12){2}{\line(1,0){0.22}}
\multiput(35,37.32)(0.22,0.12){2}{\line(1,0){0.22}}
 \multiput(34.57,37.07)(0.21,0.13){2}{\line(1,0){0.21}}
\multiput(34.15,36.8)(0.21,0.13){2}{\line(1,0){0.21}}
 \multiput(33.73,36.52)(0.21,0.14){2}{\line(1,0){0.21}}
\multiput(33.33,36.24)(0.2,0.14){2}{\line(1,0){0.2}}
 \multiput(32.92,35.94)(0.2,0.15){2}{\line(1,0){0.2}}
\multiput(32.53,35.64)(0.13,0.1){3}{\line(1,0){0.13}}
 \multiput(32.14,35.32)(0.13,0.11){3}{\line(1,0){0.13}}
\multiput(31.77,35)(0.13,0.11){3}{\line(1,0){0.13}}
 \multiput(31.4,34.66)(0.12,0.11){3}{\line(1,0){0.12}}
\multiput(31.04,34.32)(0.12,0.11){3}{\line(1,0){0.12}}
 \multiput(30.68,33.96)(0.12,0.12){3}{\line(1,0){0.12}}
  \multiput(30.34,33.6)(0.11,0.12){3}{\line(0,1){0.12}}
 \multiput(30,33.23)(0.11,0.12){3}{\line(0,1){0.12}}
\multiput(29.68,32.86)(0.11,0.13){3}{\line(0,1){0.13}}
 \multiput(29.36,32.47)(0.11,0.13){3}{\line(0,1){0.13}}
\multiput(29.06,32.08)(0.1,0.13){3}{\line(0,1){0.13}}
 \multiput(28.76,31.67)(0.15,0.2){2}{\line(0,1){0.2}}
\multiput(28.48,31.27)(0.14,0.2){2}{\line(0,1){0.2}}
 \multiput(28.2,30.85)(0.14,0.21){2}{\line(0,1){0.21}}
\multiput(27.93,30.43)(0.13,0.21){2}{\line(0,1){0.21}}
 \multiput(27.68,30)(0.13,0.21){2}{\line(0,1){0.21}}
\multiput(27.44,29.57)(0.12,0.22){2}{\line(0,1){0.22}}
 \multiput(27.2,29.12)(0.12,0.22){2}{\line(0,1){0.22}}
\multiput(26.98,28.68)(0.11,0.22){2}{\line(0,1){0.22}}
 \multiput(26.77,28.23)(0.11,0.23){2}{\line(0,1){0.23}}
\multiput(26.57,27.77)(0.1,0.23){2}{\line(0,1){0.23}}
 \multiput(26.38,27.31)(0.09,0.23){2}{\line(0,1){0.23}}
\multiput(26.21,26.84)(0.18,0.47){1}{\line(0,1){0.47}}
 \multiput(26.04,26.37)(0.16,0.47){1}{\line(0,1){0.47}}
\multiput(25.89,25.9)(0.15,0.47){1}{\line(0,1){0.47}}
 \multiput(25.75,25.42)(0.14,0.48){1}{\line(0,1){0.48}}
\multiput(25.62,24.94)(0.13,0.48){1}{\line(0,1){0.48}}
 \multiput(25.5,24.45)(0.12,0.48){1}{\line(0,1){0.48}}
\multiput(25.4,23.96)(0.1,0.49){1}{\line(0,1){0.49}}
 \multiput(25.3,23.47)(0.09,0.49){1}{\line(0,1){0.49}}
\multiput(25.22,22.98)(0.08,0.49){1}{\line(0,1){0.49}}
 \multiput(25.16,22.49)(0.07,0.49){1}{\line(0,1){0.49}}
\multiput(25.1,21.99)(0.06,0.5){1}{\line(0,1){0.5}}
 \multiput(25.06,21.49)(0.04,0.5){1}{\line(0,1){0.5}}
\multiput(25.02,21)(0.03,0.5){1}{\line(0,1){0.5}}
 \multiput(25.01,20.5)(0.02,0.5){1}{\line(0,1){0.5}}
\multiput(25,20)(0.01,0.5){1}{\line(0,1){0.5}}
 \thinlines \put(35,36){\line(0,1){2}}
\put(61,31){\line(0,1){2}}
 \put(36,35){$\alpha$}
  \put(58,29){$\beta$}
  \put(3,22){\makebox(0,0)[cc]{$\infty$}}
\put(25,19){\line(0,1){2}}
 \put(65,19){\line(0,1){2}}
 \put(25,17){\makebox(0,0)[cc]{$P_l$}}
\put(65,17){\makebox(0,0)[cc]{$P_r$}}

\end{picture}
\caption{The path of $\alpha(\tau)$ for $\RE \alpha < \RE \beta$.}
\end{center}
\end{figure}

\textit{Case 3.} $\RE \alpha = \RE \beta$. Then it follows from \eqref{ZWEI} and \eqref{DREI} that
in the case $\IM \beta < \IM \alpha$
\[
 Q'(x) <0, \quad x > P;  \quad Q'(x) > 0, \quad x < P,
\]
and in the case $\IM \beta > \IM \alpha$
\[
 Q'(x) >0, \quad x > P;  \quad Q'(x) < 0, \quad x < P.
\]
The direction of the path is indicated in each of these cases in Figure 10.

\begin{figure}[hbt]\label{nocircles}
\begin{center}
\begin{picture}(105,32)(0,13)

\thinlines \put(0,20){\line(1,0){50}}

 \put(24,27){\line(2,0){2}}
 \put(23,27){\makebox(0,0)[cc]{$\beta$}}
 \put(24,33){\line(2,0){2}}
 \put(23,33){\makebox(0,0)[cc]{$\alpha$}}
 \thicklines
                     \put(50,20){\vector(-1,0){15}}
                       \put(35,20){\line(-1,0){10}}
                       \put(25,20){\vector(0,1){4}}
                       \put(25,24){\vector(0,1){13}}
                       \put(25,37){\line(0,1){3}}

\thinlines
 \put(55,20){\line(1,0){50}}

 \put(79,27){\line(2,0){2}}
 \put(78,27){\makebox(0,0)[cc]{$\alpha$}}
 \put(79,33){\line(2,0){2}}
 \put(78,33){\makebox(0,0)[cc]{$\beta$}}
  \thicklines
                     \put(80,40){\vector(0,-1){4}}
                       \put(80,36){\vector(0,-1){13}}
                       \put(80,23){\line(0,-1){3}}
                       \put(70,20){\line(-1,0){15}}
                       \put(80,20){\vector(-1,0){10}}

 \put(25,17){\makebox(0,0)[cc]{$P$}}
 \put(80,17){\makebox(0,0)[cc]{$P$}}
\end{picture}
\caption{The path of $\alpha(\tau)$ for $\RE \alpha =\RE \beta$ when $\IM \alpha > \IM \beta$ and
when $\IM \beta > \IM \alpha$.}
\end{center}
\end{figure}

\begin{remark}
It follows from the above three cases that the point $\infty$ belongs to $\mathcal{F}_Q$ if and only if $\RE
\alpha \le \RE \beta$. This can be seen directly. It follows from \eqref{x_00} 
that $Q_\tau(z)$ has $\infty$ as GZNT if and only if
\[
\lim_{z \wh\to \infty} z Q_\tau(z) \in [0,\infty).
\]
In this case  $Q_\tau(z) \to 0$ when $z \wh\to  \infty$. Since $Q(z) \to 1$ as $z \wh\to \infty$ this can only take place for $\tau =1$, in which case
\[
\lim_{z \wh\to \infty} z Q_\tau(z)=\RE \beta -\RE \alpha,
\]
from which the assertion follows.
\end{remark}

\begin{remark}
The present treatment of the function $Q(z)$ in \eqref{einz-}  can be seen as an illustration of Theorem \ref{mainth} and Theorem \ref{realpath}.
For this purpose assume (without loss of generality) that $\alpha \in \dR$. When $\RE \beta < \alpha$ the point $P_r$ coincides with $\alpha$ and when $\RE \beta > \alpha$ the point $P_l$ coincides with $\alpha$; if $\RE \beta=\alpha$, then
 the point $P$ coincides with $\alpha$ (and $\IM \beta > 0$).
Hence the path hits $\dR$ at $\alpha$ in a perpendicular way (cf. Theorem \ref{mainth}). Afterwards it stays on the real line  (cf. Theorem \ref{realpath} with $M(z)=1 \in \mathbf{N}_0$). 
\end{remark}

\section{Path entirely on the extended real line}\label{ontherealline}

There exist functions $Q(z) \in \mathbf{N}_1$ whose GZNT $\alpha(\tau)$ stays on the (extended) real line.
These functions will be classified in this section.

\begin{example}
Let $\alpha, \beta \in \dR$ and $c \in \dR$ such that $c(\alpha-\beta) < 0$ and define the function $Q(z)$ by
\begin{equation}\label{R0}
 Q(z)=c\,\frac{z-\alpha}{z-\beta}, 
  \quad z \neq \beta. 
\end{equation}
Then $Q(z)$ belongs to  $\mathbf{N}_1$;  actually,  one has
\begin{equation}\label{R0+}
 Q(z) =\frac{(z-\alpha)^2}{(z-\beta)^2}\,\,M(z),
 \quad
 M(z)=c + \frac{c(\alpha-\beta)}{z-\alpha},
\end{equation}
where $M(z)$ belongs to $\mathbf{N}_0$

It follows from \eqref{R0+} that
$M(\alpha+)=\infty$, $M(\alpha-)=-\infty$.
Hence, according to Theorem \ref{realpath}, there exists $\varepsilon > 0$ such that the intervals
$(\alpha-\varepsilon, \alpha)$ and $(\alpha, \alpha+\varepsilon)$ belong to $\cF_Q$. However, in this case one can obtain stronger results. 
Observe that the GZNT  $\alpha(\tau)$ of $Q_\tau(z)$, $\tau \in \dR \cup \{\infty\}$, is given by
\[
 \alpha(\tau)=\frac{c \alpha-\tau \beta}{c-\tau},
 \quad \tau \in \dR\setminus\left\{c\right\},
\]
and, in the remaining cases, by 
\[
 \alpha(\infty)=\beta, \quad \alpha(c)=\infty. 
\]
Therefore the path  $\mathcal{F}_Q$ covers the extended real line.
Moreover, note that for the functions 
\begin{equation}\label{Rc}
Q(z) =\frac{d}{z-\gamma}, 
\quad \gamma\in\dR,\quad  d>0,
\end{equation}
and
\begin{equation}\label{R1c}
Q(z)=\frac{\gamma-z}{d}, 
\quad \gamma\in\dR,\quad d>0,
\end{equation}
one has  $\mathcal{F}_Q=\dR\cup\{\infty\}$ as well.
\end{example}

In fact, the functions in \eqref{R0}, \eqref{Rc}, and \eqref{R1c} are the only functions in $\mathbf{N}_1$
whose path of the GZNT stays on the (extended) real line. In order to prove this, the following lemma is
needed.

\begin{lemma}\label{kuhkuh}
Let $\al, \be \in \dR$ with $\al  \not=\be$.  Assume that $U(z)$ and $V(z)$ are nontrivial 
Nevanlinna functions which
satisfy one of the following relations 
\begin{equation}\label{q1q0}
U(z)=-\frac{(z-\al)^2 }{(z-\be)^{2}}\,V(z),
\end{equation}
\begin{equation}\label{q1q0a}
U(z)=-\frac{1}{(z-\be)^{2}}\,V(z),
\end{equation}
or
\begin{equation}\label{q1q0b}
U(z)=-(z-\al)^2 \,V(z).
\end{equation}
 Then the functions $U(z)$ and $V(z)$ are of the form
\begin{equation}\label{q1q0rep}
U(z)=c\left(1-\frac{\be-\al}{\be-z}\right), \quad V(z)=c\left(-1+\frac{\al-\be}{\al-z}\right), \quad
c(\be-\al) < 0,
\end{equation}
\begin{equation}\label{q1q0repa}
U(z)=\frac{d}{\be-z}\, \quad V(z)=d (z-\beta), \quad d>  0,
\end{equation}
or
\begin{equation}\label{q1q0repb}
U(z)= e (z-\alpha), \quad V(z)= \frac{e}{\alpha-z}, \quad e> 0,
\end{equation}
respectively. Conversely, all functions of the form \eqref{q1q0rep}, \eqref{q1q0repa}, or \eqref{q1q0repb}
are Nevanlinna functions and they satisfy \eqref{q1q0}, \eqref{q1q0a}, and \eqref{q1q0b}, respectively.
\end{lemma}

\begin{proof}
Assume that $U(z)$ and $V(z)$ are Nevanlinna functions. Let the integral representations of $U(z)$ and $V(z)$ be of the
form \eqref{nev'} with spectral measures $\sig_1$, $\sig_0$,  and constants $a_1 \in \dR$, $b_1 \ge 0$, $a_0
\in \dR$, $b_0 \ge 0$, respectively.

Assume that $U(z)$ and $V(z)$  satisfy \eqref{q1q0}. Then the identity \eqref{nev+} implies
that 
$$
b_1=\lim_{z\wh\to\infty}\frac{U(z)}{z} =-\lim_{z\wh\to\infty}\frac{V(z)}{z}  =-b_0.
$$
Since both $b_0$ and $b_1$ are nonnegative, it follows that
\begin{equation}\label{b0b1=0}
b_0=b_1=0.
\end{equation}
The Stieltjes inversion formula implies that
$$
(s-\be)^2 d\sig_1(s)=-(s-\al)^2 d\sig_0(s).
$$
It follows that $\supp \sig_1\subset\{\be\}$ and $\supp \sig_0\subset\{\al\}$. If $\sig_1=0$ then by
\eqref{b0b1=0} the function $V(z)$ is equal to a real constant, and the identity \eqref{q1q0} implies that
$U(z)=V(z)=0$. 
The same conclusion holds if $\sig_0=0$. Therefore, it may be assumed that $U(z)$ and $V(z)$
have representations of the form
$$
U(z)=a_1+\frac{d_1}{\be-z}, \quad V(z)=a_0+\frac{d_0}{\al-z},
$$
with $d_0,~d_1>0$. The identity \eqref{q1q0} implies further that $U(\al)=V(\be)=0$, so that
$$
a_1+\frac{d_1}{\be-\al}=a_0+\frac{d_0}{\al-\be}=0.
$$
Thus  $a_1(\be-\al) < 0$, $a_0(\be-\al) >  0$, and
$$
U(z)=a_1\left(1-\frac{\be-\al}{\be-z}\right), \quad V(z)=a_0\left(1-\frac{\al-\be}{\al-z}\right).
$$
This and $\eqref{q1q0}$ imply that 
$$
a_1=\lim_{z\wh\to\infty}U(z)=-\lim_{z\wh\to\infty}V(z)=-a_0,
$$
and the representations in $\eqref{q1q0rep}$ follow with $c=a_1$.

Now assume that $U(z)$ and $V(z)$  satisfy \eqref{q1q0a}. The identity \eqref{nev+}  
implies that
$$
b_1=\lim_{z \wh\to\infty}\frac{U(z)}{z} = 0.
$$
The Stieltjes inversion formula implies that
$$
(s-\be)^2 d\sig_1(s)=-d\sig_0(s),
$$
which leads to $\supp \sig_1\subset\{\be\}$ and $\supp \sig_0= \emptyset$. 
Therefore, it follows that $U(z)$ and $V(z)$ have representations of the form
$$
U(z)=a_1+\frac{d_1}{\be-z}, \quad V(z)=b_0z+a_0,
$$
with $d_1>0$.  The identity \eqref{q1q0a} implies that $V(\beta)=0$, so that $b_0 \beta+a_0=0$.  Hence
$V(z)=b_0(z-\beta)$ and, again by \eqref{q1q0a}, 
\[
 a_1+\frac{d_1}{\beta-z}=-\frac{b_0}{z-\beta}.
\]
Hence $a_1=0$ and
\[
 U(z)=\frac{d}{\beta-z}, \quad V(z)=-(z-\beta)^2 U(z)=d (z-\beta), \quad d>0.
\]

The treatment of the case where  the functions $U(z)$ and $V(z)$  satisfy \eqref{q1q0b} is similar to what has been
just shown. It also follows by symmetry from the previous case.

As to the converse statement: it is straightforward to check that the functions in  $\eqref{q1q0rep}$,
$\eqref{q1q0repa}$, or $\eqref{q1q0repb}$, are Nevanlinna functions which satisfy the identities
\eqref{q1q0}, \eqref{q1q0a}, or \eqref{q1q0b}, respectively.
\end{proof}

The   case $\alpha=\beta$ excluded in Lemma \ref{kuhkuh} concerns Nevanlinna functions
$U(z)$ and $V(z)$ which satisfy $U(z)=-V(z)$. Of course, this implies that $U(z)=c$ and $V(z)=-c$, where $c$
is a real constant.

\begin{theorem}\label{ontheline:t}
Let $Q(z) \in \mathbf{N}_1$ have the property that the path $\mathcal{F}_Q$ is contained in the extended real
line.
\begin{enumerate}\def\labelenumi{\rm (\roman{enumi})}

\item If both the GZNT and the GPNT of $Q(z)$ are finite, then $Q(z)$ is of the form
\eqref{R0} with some $\al,\be\in\dR$ and $c\in\dR$ such that  $c(\alpha-\beta) < 0$.

\item If the GZNT of $Q(z)$ is at $\infty$ and the GPNT of $Q(z)$ is finite,
then $Q(z)$ is of the form \eqref{Rc}. 

\item If the GZNT of $Q(z)$ is finite and the GPNT of $Q(z)$ is at $\infty$,
then $Q(z)$ is of the form \eqref{R1c}.  
\end{enumerate}
\end{theorem}

\begin{proof}
(i) 
The assumption that $Q_{\tau}(z)$ has no GZNT in $\dC^+$ for all $\tau\in\dR\cup
\{\infty\}$ implies that $\Im Q(z)\not=0$ in $\dC^+$; cf. Corollary \ref{quh}.  Since $Q(z) \in\mathbf{N}_1$,
there is at least one point $z_0\in\dC^+$ with $\Im Q(z_0)<0$. However,  $\Im Q(z)$ is a continuous function
on $\dC^+$, so that $\Im Q(z)<0$ on $\dC^+$. Observe that the functions $U(z)=-Q(z)$ and $V(z)=M(z)$ are
nontrivial Nevanlinna functions which satisfy $\eqref{q1q0}$. Therefore, by Lemma \ref{kuhkuh}
\[
 Q(z)=-U(z)=c \frac{z-\alpha}{z-\beta},
\]
which completes the proof of (i). 

(ii) \& (iii) The remaining parts of the theorem can be proved in a similar way, using the
respective parts of Lemma \ref{kuhkuh}. 
\end{proof}


\begin{thebibliography}{20}

\bibitem{BDHS}
J. Behrndt, V.A. Derkach, S. Hassi, and H.S.V. de Snoo, "A realization theorem for generalized Nevanlinna
families", Operators and Matrices, to appear.

\bibitem{BHSWW}
J. Behrndt, S. Hassi, H.S.V. de Snoo, R. Wietsma, and H. Winkler, "Fractional linear transforms of Stieltjes
functions", in preparation.

\bibitem{DHS1}
        V.~Derkach, S.~Hassi, and H.S.V.~de~Snoo,
        ''Operator models associated with Kac subclasses of
        generalized Nevanlinna functions'',
        Methods of Functional Analysis and Topology, 5 (1999), 65--87.

\bibitem{DHS3}
        V.A.~Derkach, S.~Hassi, and H.S.V.~de~Snoo,
        ''Rank one perturbations in a Pontryagin space with one
        negative square'', J. Functional Analysis, 188 (2002), 317-349.

\bibitem{DLLSh}
        A.~Dijksma, H.~Langer, A.~Luger, and Yu.~Shondin,
        ''A factorization result for generalized Nevanlinna
        functions of the class ${\mathbf{N}}_\kappa$'',
        Integral Equations Operator Theory, 36 (2000), 121-125.

\bibitem{donoghue}
W.F.~Donoghue, {\it Monotone matrix functions and analytic continuation}, Springer-Verlag, New
York-Heidelberg, 1974.
 
\bibitem{HSSW}
S.~Hassi, A.~Sandovici, H.S.V. de~Snoo, and H.~Winkler, ''One-dimensional perturbations, asymptotic
expansions, and spectral gaps", Oper. Theory Adv. Appl., 188 (2008), 149--173.

\bibitem{H}
E.~Hille, \textit{Analytic function theory, Volume I}, Blaisdell Publishing Company, New York, Toronto,
London, 1965.

\bibitem{JL83}
        P.~Jonas and H.~Langer, ''Some questions in the perturbation theory of
        $J$-nonnegative operators in Krein spaces'', Math. Nachr., 114 (1983),
        205--226.

\bibitem{JL95}
        P.~Jonas and H.~Langer, ''Selfadjoint extensions of a closed linear
        relation of defect one in a Krein space'', Oper. Theory Adv. Appl., 80 (1995),
        176--205.

\bibitem{KK}
I.S.~Kac and M.G.~Kre\u{\i}n, "$R$-functions -- analytic functions mapping the upper halfplane into itself'',
Supplement to the Russian edition of F.V.~Atkinson, \textit{Discrete and continuous boundary problems}, Mir,
Moscow 1968 (Russian) (English translation: Amer. Math. Soc. Transl. Ser. 2, 103 (1974), 1--18).

\bibitem{KL71}
        M.G.~Kre\u{\i}n and H.~Langer,
         ''The defect subspaces and generalized resolvents of a Hermitian operator in the space $\Pi _{\kappa
         }$'',
  Funkcional. Anal. i Prilo\u zen.  5  (1971), 54--69 (Russian).


\bibitem{KL73}
        M.G.~Kre\u{\i}n and H.~Langer,
        ''\"Uber die $Q$-function eines $\pi$-hermiteschen Operators
        in Raume $\Pi_{\kappa}$'',
        Acta. Sci. Math. (Szeged), 34 (1973), 191--230.

\bibitem{KL77}
        M.G.~Kre\u{\i}n and H.~Langer,
        ''\"Uber einige Fortsetzungsprobleme, die eng mit der Theorie
        hermitescher operatoren im Raume $\Pi _\kappa $ zusammenhangen.
        I. Einige Funktionenklassen und ihre Dahrstellungen'',
        Math. Nachr., 77 (1977), 187--236.

\bibitem{KL81}
        M.G.~Kre\u{\i}n and H.~Langer,
        ''Some propositions on analytic matrix functions related to
        the theory of operators in the space $\Pi _\kappa $'',
        Acta Sci. Math. (Szeged), 43 (1981), 181--205.

\bibitem{L}
        H.~Langer,
        ''A characterization of generalized zeros of negative type
        of functions of the class $\mathbf{N}_{\kappa}$'',
        Oper. Theory Adv. Appl., 17 (1986), 201--212.

\bibitem{MMRR1}
C. Mehl, V. Mehrmann, A.C.M. Ran, and L. Rodman,
"Eigenvalue perturbation theory of classes of structured matrices
under generic structured rank one perturbations",
Linear Algebra Appl., (2010), doi:10.1016/j.laa.2010.07.025

\bibitem{MMRR2}
C. Mehl, V. Mehrmann, A.C.M. Ran, and L. Rodman,
"Perturbation theory of selfadjoint matrices and
sign characteristics under generic structured rank one perturbations",
Linear Algebra Appl., (2010), doi:10.1016/j.laa.2010.04.008

\bibitem{RanWojtylak} A.C.M. Ran and M. Wojtylak,
"Eigenvalues of rank one perturbations of unstructured matrices",
in preparation.

\end{thebibliography}
\end{document}